\documentclass[12pt]{article} 

\usepackage[margin=1in]{geometry}
\usepackage[utf8]{inputenc}
\usepackage[dvipsnames]{xcolor}

\usepackage{graphicx}
\usepackage{amsmath,amssymb,amsthm,mathtools}
\usepackage{MnSymbol}
\usepackage{latexsym, verbatim}
\usepackage{appendix}
\usepackage{units}
\usepackage{amsmath}
\usepackage[normalem]{ulem}
\usepackage[hidelinks]{hyperref}
\usepackage{algorithm}
\usepackage{algpseudocode}
\usepackage{soul}
\usepackage{nicematrix}
\usepackage{pgfplots}
\pgfplotsset{compat=1.18}
\usetikzlibrary{pgfplots.groupplots}

\mathtoolsset{showonlyrefs}
\newtheorem{theorem}{Theorem}[section]
\newtheorem{corollary}[theorem]{Corollary} 
\newtheorem{lemma}[theorem]{Lemma}
\newtheorem{proposition}[theorem]{Proposition}

\newtheorem{remark}[theorem]{Remark}

\theoremstyle{definition}
\newtheorem{example}[theorem]{Example}
\newtheorem{definition}[theorem]{Definition}

\newtheorem*{theorem**}{Theorem\theoremnum}
\newenvironment{theorem*}[1][]{%
  \edef\theoremnum{\if\relax\detokenize{#1}\relax\else~#1\fi}
  \begin{theorem**}
}{%
  \end{theorem**}
}

\newcommand{\fa}{\mathbf{a}}
\newcommand{\fb}{\mathbf{b}}
\newcommand{\fx}{\mathbf{x}}

\newcommand{\Ucal}{\mathcal{U}}
\newcommand{\Wcal}{\mathcal{W}}

\newcommand{\Vcal}{\mathcal{V}}

\newcommand{\CC}{\mathbb{C}}

\newcommand{\PP}{\mathbb{P}}
\newcommand{\RR}{\mathbb{R}}

\newcommand{\ideal}{\mathcal{I}}
\DeclareMathOperator{\Gr}{Gr}
\DeclareMathOperator{\Span}{Span}
\DeclareMathOperator{\rank}{rank}

\newcommand{\T}{^\mathsf{T}}

\newcommand{\Zcal}{\mathcal{Z}}

\title{{\bf Identifiability of overcomplete \\ independent component analysis}}
\author{Kexin Wang and Anna Seigal}
\date{}

\begin{document}
\maketitle

\begin{abstract}
Independent component analysis (ICA) studies mixtures of independent latent sources. 
An ICA model is identifiable if the mixing can be recovered uniquely. 
It is well-known that ICA is identifiable if and only if at most one source is Gaussian. 
However, this applies only to the setting where the number of sources is at most the number of observations. In this paper, we generalize the identifiability of ICA to the overcomplete setting, where the number of sources exceeds the number of observations. 
We give an if and only if characterization of the identifiability of overcomplete ICA.
The proof studies linear spaces of rank one symmetric matrices.
For generic mixing, we present an identifiability condition in terms of the number of sources and the number of observations.
We use our identifiability results to design an algorithm to recover the mixing matrix from data and apply it to synthetic data and two real datasets.
\end{abstract}

\section{Introduction}

Blind source separation seeks to recover latent sources and unknown mixing from observations of mixtures of signals~\cite{comon2010handbook}.
A special case is independent component analysis (ICA), which assumes that the latent sources are independent. 
Classical ICA assumes that the observations are linear mixtures of the independent sources. That is,
\begin{equation}
\label{eqn:defica}
\mathbf{x}=A\mathbf{s},
\end{equation}
where $\mathbf{s} = (s_1,\ldots,s_J)\T$ is a vector of independent sources, $\mathbf{x}=(x_1,\ldots,x_I)\T$ collects the observed variables, 
and $A \in \RR^{I \times J}$ is an unknown mixing matrix.
Applications of ICA include recovering speech signals~\cite{bartlett2002face} and
brain signals~\cite{jung2001imaging}, casual discovery~\cite{shimizu2006linear}, and image decomposition~\cite{hyvarinen1999fast,podosinnikova2019overcomplete}.
The ICA framework has seen extensions to nonlinear mixtures, see e.g.~\cite{jutten2004advances}.

The ICA model~\eqref{eqn:defica} is identifiable
if the mixing matrix $A$ can be uniquely recovered from $\fx$, up to column scaling and permutation. 
Identifiability is crucial to interpreting the entries of the mixing matrix. Depending on the application, these encode casual relationships~\cite{shimizu2006linear} or image components~\cite{hyvarinen1999fast,podosinnikova2019overcomplete}. 
Scaling and permutation indeterminacy are unavoidable, corresponding to the arbitrary order and scale of the sources, which does not affect their independence\footnote{ 
Let $B = A PD$, for permutation matrix $P$ and diagonal matrix $D$.
Then 
$A \mathbf{s} = B \mathbf{r}$,
where $\mathbf{r} = D^{-1} P\T \mathbf{s}$ re-orders and scales $\mathbf{s}$. 
}.
 
The following characterization of the identifiability of ICA is well-known. Recall that a distribution is non-degenerate if it is not supported at a single point.

\begin{theorem}[{\cite[Theorem 11 and Corollary 13]{comon1994independent}}]
\label{thm:comon}
Consider the ICA model $\mathbf{x}=A\mathbf{s}$, where $\mathbf{s} = (s_1,\ldots,s_I)\T$ is a vector of non-degenerate independent sources, $\mathbf{x}=(x_1,\ldots,x_I)\T$ is a vector of observations, and $A \in \RR^{I \times I}$ is invertible. Identifiability holds if and only if at most one of the sources is Gaussian. 
\end{theorem}

Theorem~\ref{thm:comon} stems from the connection between linear transformations of independent variables and Gaussianity.

\begin{theorem}[{The Darmois–Skitovich theorem~\cite{darmois1953analyse,skitovitch1953property,skitovivc1962linear}}]
    Let $s_1,\ldots,s_I$ be non-degenerate independent random variables. If the linear combinations $\sum_{i=1}^I \lambda_is_i$ and $\sum_{j=1}^I \mu_j s_j$ are independent, then any $s_i$ with $\lambda_i\mu_i\neq 0$ is Gaussian. 
\end{theorem}

Theorem~\ref{thm:comon} resolves the identifiability of ICA when the number of sources and observations are equal, 
the case of square mixing matrix $A \in \RR^{I \times I}$. It extends to the case of fewer sources than observations, provided the mixing matrix has full rank, see~\cite[Theorem 3]{1306473}. 

\medskip 

Our goal in this paper is to give a characterization of the identifiability of ICA that does not restrict the number of sources and observations. That is, we seek to generalize Theorem~\ref{thm:comon} to overcomplete ICA, where there are more sources than observations. Overcomplete ICA appears in sparse coding
and finding signals in speech data~\cite{lewicki2000learning}, as well as decomposing images~\cite{olshausen1995sparse,hyvarinen1999fast}. 
Algorithms for overcomplete ICA include~\cite{davies2004simple,theis2004geometric,podosinnikova2019overcomplete}.

To date, a characterization of the identifiability of overcomplete ICA has been missing. 
The following partial results are known. 
If no source is Gaussian, then~\eqref{eqn:defica} is identifiable if and only if no pair of columns in $A$ are collinear~\cite[Theorem 3]{1306473}.
If there are at least two Gaussian sources, non-identifiability holds like in the square case, as follows.  
Suppose that sources $s_1$ and $s_2$ are  Gaussian, with variances $\sigma_1$ and $\sigma_2$, respectively. Let $u_1=\lambda_1 s_1+ \lambda_2 s_2$ and $u_2= \mu_1 s_1+\mu_2 s_2$. The variables $u_1,u_2$ are Gaussian, hence they are independent if and only if they are uncorrelated. Fix non-zero $\lambda_i$ and $\mu_j$ such that $\mathbb{E}[u_1u_2]-\mathbb{E}[u_1]\mathbb{E}[u_2]=\lambda_1 \mu_1 \sigma_1+ \lambda_2 \mu_2 \sigma_2=0$ and define $\nu = \lambda_1 \mu_2 - \lambda_2 \mu_1$. Let $\mathbf{r}:=(\nu u_1, \nu u_2,s_3,\ldots,s_J)\T$. Then $A \mathbf{s}$ and $ B \mathbf{r}$ have the same distribution, where
$$
B = \begin{pmatrix}
\vdots&\vdots& \vdots & &\vdots\\
 \mu_2 \mathbf{a}_1-\mu_1 \mathbf{a}_2 & -\lambda_2 \mathbf{a}_1+\lambda_1 \mathbf{a}_2 & \mathbf{a}_3 & \cdots&\mathbf{a}_J\\
\vdots&\vdots& \vdots & & \vdots
\end{pmatrix},$$
and $\mathbf{a}_1 , \ldots, \mathbf{a}_J$ are the columns of $A \in \RR^{I \times J}$.
Matrices $A$ and $B$ are not the same up to permutation and scaling, hence identifiability does not hold.

To characterize the identifiability of overcomplete ICA, it remains to settle the case where a single source is Gaussian. 
Since identifiability is impossible with more than two Gaussian sources, we make the following definition.

\begin{definition}
\label{def:id}
A mixing matrix $A \in \RR^{I\times J}$ is \emph{identifiable} if for any non-degenerate sources $\mathbf{s}=(s_1,
\ldots,s_J)\T$ with at most one $s_j$ Gaussian, the matrix $A$ can be recovered uniquely, up to permutation and scaling of its columns, from $A \mathbf{s}$. That is, if $A \mathbf{s}$ and $B \mathbf{r}$ have the same distribution for some $B \in \RR^{I \times K}$ with $K \leq J$ and some $\mathbf{r} = (r_1, \ldots, r_K)$, with 
the same number of Gaussian entries as $\mathbf{s}$, then $J = K$ and matrices $A$ and $B$ coincide, up to permutation and scaling of columns. 
\end{definition}

\begin{remark}[Recovering mixing vs. sources]
In casual inference~\cite{shimizu2006linear}, the mixing matrix reveals the causal relationships between variables, while the source variables give the distributions of the exogenous noise. In image decomposition~\cite{hyvarinen1999fast,podosinnikova2019overcomplete}, the mixing matrix gives the image components, while the source variables follow Bernoulli distributions.
    One can recover the sources from the mixing $A$, provided $A$ has a left inverse. This holds for full rank matrices $A \in \RR^{I \times J}$ with $J \leq I$. Once $J > I$, we can no longer recover the sources, but may still recover the mixing matrix.
\end{remark}

 In this paper, identifiability refers to Definition~\ref{def:id}.
Given vector $\mathbf{v} \in \RR^I$, the rank one matrix $\mathbf{v} \mathbf{v}\T$ is denoted $\mathbf{v}^{\otimes 2}$.
Our first contribution is an if and only if characterization of the identifiability of ICA, with no restrictions on the number of sources or observations. 

\begin{theorem}\label{thm: iff for identifiablity}
Fix $A\in \RR^{I\times J}$ with columns $\fa_{1},\ldots, \fa_{J}$ and no pair of columns collinear. Then $A$ is identifiable if and only if the linear span of $\fa_1^{\otimes 2},\ldots, \fa_J^{ \otimes 2}$ does not contain any real matrix $\mathbf{b}^{\otimes 2}$ unless $\mathbf{b}$ is collinear to $\fa_j$ for some $j \in \{ 1, \ldots, J\}$.
\end{theorem}

\begin{figure}[htbp]
\centering
    \includegraphics[width=8cm]{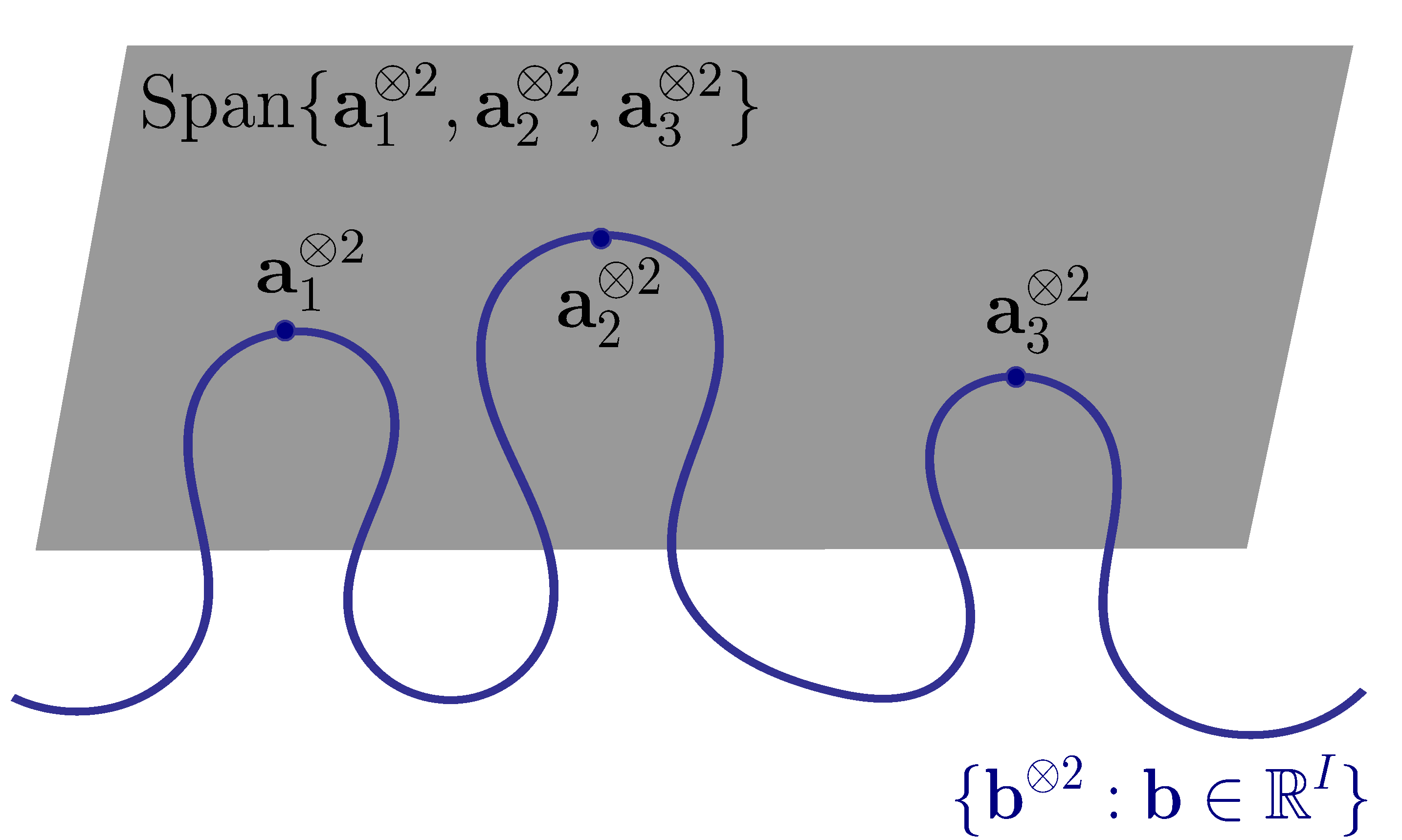}
    \caption{Illustration of Theorem~\ref{thm: iff for identifiablity}}
\end{figure}

The $d$-th cumulant of a distribution on $\RR^I$ is a symmetric order $d$ tensor of format $I \times \cdots \times I$ that encodes properties of the distribution~\cite[Chapter 2]{mccullagh2018tensor}. 
The $d$-th cumulant of $\mathbf{x} = A \mathbf{s}$ is 
$$ \kappa_d=\sum\limits_{j=1}^J \lambda_j \mathbf{a}_j^{\otimes d},$$
where $\mathbf{s} = (s_1, \ldots, s_J)$ has independent entries, the scalar $\lambda_j$ is the $d$-th cumulant of $s_j$, and $\mathbf{a}_j^{\otimes d}$ is the tensor with $(i_1, \ldots, i_d)$ entry $(\mathbf{a}_j)_{i_1} \cdots (\mathbf{a}_j)_{i_d}$.
This follows from the fact that the cumulant tensor of a vector of independent entries is diagonal and from the multilinearity property of cumulants.

Theorem~\ref{thm: iff for identifiablity} may be surprising at first, since it only uses second order information about the matrix $A$. We might have expected a condition that involves terms $\mathbf{a}_j^{\otimes d}$ from higher-order cumulants.
However, since for a Gaussian
all cumulants of order greater than two are zero, our characterization turns out to only depend on the second powers $\mathbf{a}_j^{\otimes 2}$.

Theorem \ref{thm: iff for identifiablity} implies the identifiability of square ICA, as follows.

\begin{example}
Let $\fx =A \mathbf{s}$, where $A \in \RR^{I \times I}$ is invertible and $\mathbf{s}$ is a vector of independent sources with one source Gaussian. 
Assume there exists $\mathbf{b} \in \RR^I$ with 
$\mathbf{b}^{\otimes 2} = \sum_{i = 1}^I \lambda_i \mathbf{a}_i^{\otimes 2}$. After a change of basis, we have
$(\mathbf{b}^\prime)^{\otimes 2} = \sum_{i = 1}^I \lambda_i \mathbf{e}_i^{\otimes 2}$, where the $\mathbf{e}_i$ are elementary basis vectors, since the columns $\mathbf{a}_i$ are linearly independent. 
Hence $\mathbf{b}^\prime$ is a diagonal rank one matrix. Therefore $\mathbf{b}^\prime$ is parallel to $\mathbf{e}_i$ for some $i$, so $\mathbf{b}$ is parallel to $\mathbf{a}_i$ for some $i$.
The model is then identifiable, by Theorem \ref{thm: iff for identifiablity}. 
\end{example}

\noindent The following examples illustrate Theorem~\ref{thm: iff for identifiablity} in overcomplete settings.

\begin{example}
\label{ex:id}
Consider the mixing matrix
\begin{equation}
    \label{eqn:id_A}
    A=\begin{pmatrix}
    1&0&0&0&1&0\\
    0&1&0&0&1&1\\
    0&0&1&0&1&1\\
    0&0&0&1&0&1
\end{pmatrix}.
\end{equation}
No pair of columns of $A$ are collinear. Let $\mathbf{b}^{\otimes 2} = \sum_{j = 1}^6 \lambda_j \fa_j^{\otimes 2}$. 
Then 
$$ \mathbf{b}^{\otimes 2} = \begin{bmatrix} \lambda_1 + \lambda_5 & \lambda_5 & \lambda_5 & 0 \\ \lambda_5 & \lambda_2 + \lambda_5  + \lambda_6 & \lambda_5 + \lambda_6 & \lambda_6 \\ 
\lambda_5 & \lambda_5 + \lambda_6 & \lambda_3+ \lambda_5 + \lambda_6 & \lambda_6 \\ 
0 & \lambda_6 & \lambda_6 & \lambda_4 + \lambda_6 \end{bmatrix} .$$
The $2 \times 2$ minors of this matrix vanish, since $\mathbf{b}^{\otimes 2}$ has rank one. This cannot happen unless all but one $\lambda_i$ is zero, as can be seen from a Macaulay2~\cite{grayson2002macaulay2} computation, so $\mathbf{b}$ is collinear to one of the $\mathbf{a}_j$. Hence $A$ is identifiable, by Theorem~\ref{thm: iff for identifiablity}.

To explain the condition in Theorem~\ref{thm: iff for identifiablity}, 
we show directly that $A$ is identifiable. 
To simplify our exposition, we assume that the non-Gaussian sources have non-vanishing fourth cumulants. Suppose $s_1,\ldots,s_5$ are the non-Gaussian sources. A tensor of the form $\sum_{j=1}^5 \lambda_j \fa_j^{\otimes 4}$ has a unique tensor decomposition, 
by Kruskal's criterion~\cite{KRUSKAL197795}. Hence columns $\fa_1,\ldots,\fa_5$ can be recovered uniquely, up to permutation and scaling. The covariance matrix of $A\mathbf{s}$ has the form $\sum_{j=1}^6 \mu_j\fa_j^{\otimes 2}$. If there are two candidates $\mathbf{a}_6$ and $\mathbf{b}$ for the last column, then $\mathbf{b}^{\otimes 2}  \in \Span \{ \mathbf{a}_j^{\otimes 2} : j = 1, \ldots, 6 \}$.
\end{example}

\begin{example}
\label{ex:nonid}
The mixing matrix
\begin{equation}
    \label{eqn:non-id_A}
A =\begin{pmatrix}
    1&0&1\\
    0&1&1
\end{pmatrix}
\end{equation}
does not satisfy the condition in Theorem~\ref{thm: iff for identifiablity}, since 
 $\mathbf{b}^{\otimes 2} = -(\fa_1)^{\otimes 2}+2(\fa_2)^{\otimes 2}+2(\fa_3)^{\otimes 2}$ holds for $\mathbf{b} =(1,2)$. Hence $A$ 
is non-identifiable.
We exhibit the non-identifiability, as follows. 
Suppose $s_1,s_2,s_3$ follow exponential distributions with parameter~$1$, that $r_1,r_2$ follow standard Gaussian distributions, and that $s_1,s_2,s_3,r_1,r_2$ are independent. 
Then 
$$ \begin{pmatrix}
    1&0&1\\
    0&1&1
\end{pmatrix} \begin{pmatrix} s_1 \\ s_2-r_1+r_2 \\ r_1+r_2\end{pmatrix} = \begin{pmatrix}
    1&0&1\\
    0&1&2
\end{pmatrix} \begin{pmatrix} s_1+r_1 \\ s_2 \\ r_2 \end{pmatrix}. $$
 Both source vectors have independent entries with the last coordinate Gaussian, since $r_1+r_2$ and $-r_1+r_2$ are independent Gaussians.
\end{example}

Our second contribution characterizes whether a generic matrix $A \in \RR^{I\times J}$ is identifiable. A generic matrix is one that lies outside of a set defined by the vanishing of some equations. In particular, genericity holds almost surely in $\RR^{I \times J}$.

\begin{theorem}\label{thm: real generic matrix result}
Let $A \in \RR^{I \times J}$ be generic. Then
\begin{enumerate}
    \item If $J \leq {I \choose 2}$ 
    or if $(I, J) = (2, 2)$ or $(3, 4)$, then $A$ is identifiable;
    \item If $J= {I\choose 2}+1$, where $I\geq 4$ and $I \equiv 2,3 \mod 4$, then there is positive probability that $A$ is identifiable and a positive probability that $A$ is non-identifiable;
    \item If $J> {I\choose 2}+1$ or if $J= {I\choose 2}+1$ and $I\equiv 0,1 \mod 4$, then $A$ is non-identifiable. 
\end{enumerate}
\end{theorem}

To prove Theorem~\ref{thm: real generic matrix result}, we first studying the identifiability condition in Theorem~\ref{thm: iff for identifiablity} over the complex numbers. Then we specialize to the real numbers, proving the following. No extra complex solutions implies no extra real solutions. If there are a finite number $k$ of extra complex solutions, the presence of extra real solutions depends on the parity of $k$: if odd, there is an extra solution, but if even, there may or may not be. If there are infinitely many extra complex solutions, then we show that there is an extra real solution.

Our third contribution is an algorithm to recover the mixing matrix $A \in \RR^{I\times J}$, via tensor decomposition.
ICA has close connections to tensor decomposition. Comon explained how to obtain the mixing matrix from the higher order cumulant tensors of the observed variables~\cite{comon1994independent}. 
When the sources are non-Gaussian and the number of sources is at most the number of observations, tensor algorithms include JADE~\cite{cardoso1993blind}, which uses fourth order cumulants, STOTD~\cite{de2001independent}, which uses third order cumulants, and algorithms that combine cumulants of different orders~\cite{moreau2001generalization}. 
For overcomplete ICA, algorithms include FOOBI~\cite{de2007fourth}, which uses fourth order cumulants and BIRTH, which uses hexacovariance (the flattening of the six order cumulant)~\cite{albera2004blind}. 
However, to our knowledge, there do not exist algorithms that apply to the setting where one of the sources is Gaussian.
Having a Gaussian source is natural to represent noise or, in practice, to allow sources that are close to Gaussian~\cite{sokol2014quantifying}.

\begin{algorithm}[htbp]
\caption{Recover $A$ from the second and fourth order cumulants of $\mathbf{x}$}\label{alg:recover A}
\begin{algorithmic}[1]
\renewcommand{\algorithmicrequire}{\textbf{Input:}}
\Require Second and fourth order cumulants $\kappa_2,\kappa_4$ of $\mathbf{x} =A\mathbf{s}$ for matrix $A \in \RR^{I \times J}$ and independent sources $\mathbf{s}$ with one Gaussian $s_J$. 
\State \textbf{The first $J-1$ columns of $A$:} Compute the symmetric tensor decomposition of $\kappa_4$ to recover $\fa_1, \ldots, \fa_{J-1}$, up to permutation and scaling.
\State \textbf{The last column of $A$:} Find a rank one matrix in $\Span\{\fa_1^{\otimes 2},\ldots, \fa_{J-1}^{\otimes 2},\kappa_2\}$ that is not collinear to $\fa_1^{\otimes 2},\ldots, \fa_{J-1}^{\otimes 2}$, by initializing at a random rank one matrix and a random point in the span and minimizing the distance between them using Powell's method~\cite{powell1964efficient}. \renewcommand{\algorithmicrequire}{\textbf{Output:}}
\Require Matrix $A \in \RR^{I \times J}$ with columns $\mathbf{a}_1, \ldots, \mathbf{a}_J$.
\end{algorithmic}
\end{algorithm}

Our algorithm uses~\cite[Algorithm 1]{kileel2019subspace} in the first step. We could, in principle, use any tensor decomposition algorithm here. We use this method because of its compatibility with our second step: both look for rank one matrices or tensors in a linear space.

We apply our algorithm to synthetic data, where it corroborates our identifiability results, to the CIFAR-10 dataset~\cite{cifar10} of images, to incorporate Gaussian noise into bases of image patches, cf.~\cite{podosinnikova2019overcomplete}, and to protein signalling~\cite{sachs2005causal}, to incorporate Gaussian noise into the causal structure learning algorithm of~\cite{shimizu2006linear}.

The rest of the paper is organized as follows. We prove Theorem~\ref{thm: iff for identifiablity} in Section~\ref{section: symmetric rank one matrices}. We relate identifiability to systems of quadrics in Section~\ref{sec:iden to quadrics} and study these quadrics in Section~\ref{sec:quadrics}. We prove Theorem~\ref{thm: real generic matrix result} in Section~\ref{sec:complex_to_real}. We study the identifiability of special matrices in Section~\ref{section: special matrices}. Our numerical results are in Section~\ref{section: numerical result}.

\section{Characterization of identifiability}\label{section: symmetric rank one matrices}

We prove Theorem~\ref{thm: iff for identifiablity}, our characterization of the identifiability of ICA.

\subsection{Sufficiency}

We show that the condition of Theorem~\ref{thm: iff for identifiablity} is sufficient for identifiability of ICA.

\begin{proposition}\label{prop: sufficient con for identifiability}
Fix $A\in \RR^{I\times J}$ with columns $\fa_1, \ldots, \fa_J$, and no pair of columns collinear. Then $A$ is identifiable if the linear span of the rank one matrices $\fa_1^{\otimes 2},\ldots,\fa_J^{ \otimes 2}$ does not contain any real rank one matrix $\mathbf{b}^{\otimes 2}$, unless $\mathbf{b}$ is collinear to $\fa_j$ for some $j \in \{ 1, \ldots, J\}$.
\end{proposition}

For the proof, we use the second characteristic function, the cumulant generating function. We also use the following results: 
Theorem~\ref{thm:secondchar} relates the second characteristic function of a linear mixing to its sources, Theorem~\ref{thm:general identifiability result} settles the uniqueness of the mixing matrix of the non-Gaussian sources, and Theorem~\ref{thm:summandpolynomial} turns the study of second characteristic functions into degree two equations.

\begin{theorem}[{\cite[Proposition 9.4]{comon2010handbook}}]\label{thm:secondchar}
Let $\mathbf{x} = A \mathbf{s}$ with the entries of $\mathbf{s}$ independent. Then
$$
\Psi_\mathbf{x}(\mathbf{u})=\sum_{j=1}^J \Psi_{s_j}(\mathbf{u}\T \mathbf{a}_j),
$$
in a neighbourhood of the origin,
where $\Psi_\mathbf{x}$ and $\Psi_{s_j}$ are the second characteristic functions of the random variables $\mathbf{x}$ and $s_j$.
\end{theorem}

\begin{theorem}[{See~\cite[Theorem 9.1]{comon2010handbook} and~\cite[Theorem 10.3.1]{kagan1973characterization}}]
\label{thm:general identifiability result}
Let $\mathbf{x}=A\mathbf{s}$, where $s_j$ are independent and non-degenerate and $A$ does not have any collinear columns. Then $\mathbf{x}=A_1 \mathbf{s_1}+ A_2 \mathbf{s_2}$, where $\mathbf{s_1}$ is a vector of non-Gaussian, $\mathbf{s_2}$ is a vector of Gaussian and independent of $\mathbf{s_1}$, and $A_1$ is unique, up to permuting and scaling of its columns.
\end{theorem}

\begin{theorem}[{\cite[Lemma A.2.4]{kagan1973characterization}}]\label{thm:summandpolynomial}
Fix vectors $\mathbf{\alpha}_1,\ldots,\mathbf{\alpha}_n \in \RR^p$ and a vector of variables $\mathbf{u}:= (u_1, \ldots, u_p)\T$. 
Assume that $\alpha_i$ is not collinear to $\alpha_j$ for $i\neq j$ or to any elementary basis vector.
Let $\psi_1,\ldots,\psi_n,A_1,\ldots,A_p$ be complex-valued continuous functions. Assume that
$$
\psi_1(\alpha_1\T \mathbf{u})+\cdots+\psi_n(\alpha_n\T \mathbf{u})=\sum_{i=1}^p A_i(u_i)+P_k(\mathbf{u}),
$$
for all $\mathbf{u}$ in a neighborhood of the origin,
where $P_k$ is a degree $k$ polynomial in $\mathbf{u}$. Then the functions $\psi_j$ and $A_i$ are all polynomials of degree at most $\max \{ n,k\}$ in an interval around the origin.
\end{theorem}

\begin{proof}[Proof of Proposition \ref{prop: sufficient con for identifiability}]
The model $\mathbf{x}=A\mathbf{s}$ is identifiable when $\mathbf{s}$ does not contain a Gaussian, by~\cite[Theorem 3]{1306473}, since $A$ has no pair of columns collinear. 
It remains to consider the case that $\mathbf{s}$ contains one Gaussian. Without loss of generality, suppose that $s_J$ is standard Gaussian. 
Take $B \in \RR^{I \times K}$ with $K \leq J$ and $\mathbf{r} = (r_1, \ldots, r_K)$ with $B \mathbf{r} = \mathbf{x}$. The columns of $B$ corresponding to non-Gaussian sources must each be collinear to one of $\fa_1,\ldots,\fa_{J-1}$, 
by Theorem \ref{thm:general identifiability result}. 
The vectors $\mathbf{r}$ and $\mathbf{s}$ have the same number of Gaussian entries.
The number of non-Gaussian and Gaussian sources in $\mathbf{r}$ must then be $J-1$ and $1$ respectively, by Theorem \ref{thm:general identifiability result} and the assumption $K\leq J$.
Without loss of generality, assume that $r_J$ is a standard Gaussian. Then the first $J-1$ columns of $B$ equal the first $J-1$ columns of $A$, up to scaling and permutation, by Theorem \ref{thm:general identifiability result}. Denote the last column of $B$ by $\mathbf{b}$. 

The second characteristic function of a standard Gaussian $x$ is $\Psi_x(t)=-\frac{1}{2}t^2$.
We have the equality 
\begin{equation}\label{eq:secondchar}
\sum_{j=1}^{J-1} (\Psi_{s_j}-\Psi_{r_j})(\mathbf{u}\T \mathbf{a}_j)- \frac{1}{2}(\mathbf{u}\T \mathbf{a}_J)^2+ \frac{1}{2}(\mathbf{u}\T \mathbf{b})^2=0,
\end{equation}
by Theorem \ref{thm:secondchar}. 
We apply Theorem \ref{thm:summandpolynomial} to the functions $\psi_j=\Psi_{s_j}-\Psi_{r_j},A_i=0,P_k=\frac{1}{2}(\mathbf{u}\T \mathbf{a}_J)^2-\frac{1}{2}(\mathbf{u}\T \mathbf{b})^2$. It shows that $\Psi_{s_j}-\Psi_{r_j}$ is polynomial in a neighborhood of the origin, for all $1\leq j \leq J-1$.
Taking the degree two part of \eqref{eq:secondchar}, we obtain an identity 
$\sum_{j=1}^{J-1}\lambda_j(\mathbf{u}\T \mathbf{a}_j)^2+(\mathbf{u}\T \mathbf{a}_J)^2-(\mathbf{u}\T \mathbf{b})^2=0$,
for some scalars $\lambda_1, \ldots, \lambda_{J-1}$.
That is,
$$
\mathbf{b}^{\otimes 2}= \fa_J^{\otimes 2} + \sum_{j=1}^{J-1}\lambda_j \fa_j^{\otimes 2}.$$
By the condition in the statement, we conclude that $\mathbf{b}$ is collinear to $\fa_j$ for some $j=1,\ldots,J$. However, $\mathbf{b}$ is not collinear to the first $J-1$ columns of $B$, since $B$ has no pair of columns collinear. Hence, it is not collinear to the first $J-1$ columns of $A$, and therefore $\mathbf{b}$ is collinear to $\fa_J$. Hence $A$ and $B$ are equal, up to permutation and scaling of columns.
\end{proof}

Taking the degree two part of~\eqref{eq:secondchar} requires 
Theorem~\ref{thm:summandpolynomial}: in general, second characteristic functions may not have Taylor expansions.
We now show how Proposition~\ref{prop: sufficient con for identifiability} suggests the viability of Algorithm~\ref{alg:recover A}.

\begin{theorem}
Fix $J\leq {I\choose 2}$. 
Suppose we have a generic $A\in \RR^{I\times J}$ satisfying the condition in Theorem \ref{thm: iff for identifiablity} and a system of independent sources $\mathbf{s}$ with one Gaussian and the rest non-Gaussian with non-vanishing second and fourth cumulants. Then $A$ can be recovered, up to permutation and scaling of its columns, from the second and fourth cumulants of $A\mathbf{s}$.
\end{theorem}

\begin{proof}
Without loss of generality, suppose that $s_J$ is the Gaussian source. Let the fourth cumulant of $s_j$ be $\lambda_j \in \RR \backslash \{ 0 \}$. Then the fourth cumulant tensor of $\mathbf{x} = A \mathbf{s}$ is $\kappa_4=\sum_{j=1}^{J-1}\lambda_j \fa_j^{\otimes 4}$.
The rank is $J-1$, 
by the genericity of the vectors $\fa_j$.
A generic symmetric tensor of format $I\times I\times I \times I $ has symmetric rank $\lceil \frac{1}{24} (I+3)(I+2)(I+1) \rceil$, by the Alexander-Hirschowitz Theorem~\cite{alexander1995poly}.
The rank of $\kappa_4$ is less than the generic rank, since $J - 1 \leq {{ I \choose 2}}$ and the inequality ${I\choose 2} < \lceil \frac{1}{24} (I+3)(I+2)(I+1) \rceil$ holds for all $I$. 
Hence $\kappa_4$ has a unique symmetric decomposition, by~\cite[Theorem 1.1]{chiantini2017generic}. Therefore, the first $J-1$ columns of $A$ can be recovered, up to scaling and permutation, via the symmetric tensor decomposition of $\kappa_4$. 

The second cumulant of $\fx$ is $\kappa_2=\sum_{j=1}^J \mu_k \fa_j\otimes \fa_j$, where $\mu_j$ is the variance of $s_j$. Hence $\Span\{\fa_1^{\otimes 2},\ldots, \fa_{J-1}^{\otimes 2},\kappa_2\}=\Span\{\fa_1^{\otimes 2},\ldots, \fa_{J}^{\otimes 2}\}$. 
By assumption, the only real rank one matrices in $\Span\{\fa_1^{\otimes 2},\ldots, \fa_{J}^{\otimes 2}\}$ are $\fa_1^{\otimes 2},\ldots, \fa_{J}^{\otimes 2}$ up to scalar, so a rank one matrix $\fa_J^{\otimes 2}$ in $\Span\{\fa_1^{\otimes 2},\ldots, \fa_{J-1}^{\otimes 2},\kappa_2\}$ that is not collinear to $\fa_1^{\otimes 2},\ldots, \fa_{J-1}^{\otimes 2}$ must be a scalar multiple of $\fa_J^{\otimes 2}$. Hence, $A$ is recovered uniquely up to column permutation and scaling from the second and fourth cumulant tensors of $A\mathbf{s}$.
\end{proof} 

\subsection{Necessity}
We complete the proof of Theorem \ref{thm: iff for identifiablity}, showing that our condition is necessary.

\begin{proposition}\label{prop: necessary con for identifiability}
Fix matrix $A\in \RR^{I\times J}$ with columns $\fa_1, \ldots, \fa_J$. Then $A$ is identifiable only if
no pair of its columns is collinear and 
the linear span of $\fa_1^{\otimes 2},\ldots,\fa_J^{ \otimes 2}$ does not contain any real matrix $\mathbf{b}^{\otimes 2}$ unless $\mathbf{b}$ is collinear to $\fa_j$ for some $j \in \{ 1, \ldots, J\}$.
\end{proposition}

\begin{proof}
If the mixing matrix $A$ has two collinear columns, we can combine them and obtain a matrix $B\in \RR^{I\times (J-1)}$ and a system of independent sources $\mathbf{r} = (r_1, \ldots, r_{J-1})$ such that $B\mathbf{r}$ and $A\mathbf{x}$ have the same distribution. Hence identifiability implies no pair of collinear columns. 

Assume there exists $\fb^{\otimes 2} \in \Span \{ \fa_1^{\otimes 2}, \ldots, \fa_J^{\otimes 2} \}$ with $\fb$ not collinear to any $\fa_j$. Then 
$$
\fb^{\otimes 2} = \sum_{j=1}^{J-1} \lambda_j\fa_j^{\otimes 2}+\fa_J^{\otimes 2},
$$
since we can assume without loss of generality that the coefficient of $\fa_J$ is one. 
Let $B$ be the matrix with columns $\fa_1,\ldots,\fa_{J-1},\fb$.
We will construct independent random variables $s_j$ and $r_j$ such that $A\mathbf{s}$ and $B\mathbf{r}$ have the same distribution. 

Let $r_J$ and $s_J$ be standard Gaussians. Choose $J-1$ non-Gaussian random variables $y_1,\ldots,y_{J-1}$ and Gaussian distributions $z_1,\ldots,z_{J-1}$ with second characteristic functions 
$$ \Psi_{z_j}(t) = \begin{cases} -\frac{1}{2}|\lambda_j|t^2 & \lambda_j\neq 0 \\ 
0 & \lambda_j=0, \end{cases} $$
such that these random variables together with the standard Gaussians $r_J,s_J$ are independent.
When $\lambda_j\geq 0$, set $r_j=y_j$ and $s_j=y_j+z_j$; when $\lambda_j<0$, set $s_j=y_j$ and $r_j=y_j+z_j$. Then $s_1,\ldots,s_J$ are independent and $r_1,\ldots,r_J$ are independent. 
The source variables differ by Gaussians. We have
$$ ( \Psi_{s_j} - \Psi_{r_j})(t) = \begin{cases} ( \Psi_{y_j + z_j} - \Psi_{y_j}) (t) & \lambda_j > 0 \\ 
( \Psi_{y_j} - \Psi_{y_j  + z_j}) (t) & \lambda_j < 0 \\ 
0 & \lambda_j = 0. 
\end{cases}$$
All three cases evaluate to give $( \Psi_{s_j} - \Psi_{r_j})(t) = -\frac12 \lambda_j t^2$.
The second characteristic functions of $A\mathbf{s},B\mathbf{r}$ are equal, 
by Theorem \ref{thm:secondchar}, since 
\begin{align}
\Psi_{A\mathbf{s}}(\mathbf{u})-\Psi_{B\mathbf{r}}(\mathbf{u})=\sum_{j=1}^{J-1} (\Psi_{s_j}-\Psi_{r_j})(\mathbf{u}\T \fa_j)-\frac{1}{2}(\mathbf{u}\T \fa_J)^2+\frac{1}{2}(\mathbf{u}\T \mathbf{b})^2=0.
\end{align}
 Hence $A\mathbf{s}$ and $B\mathbf{r}$ have the same distribution. The last column of $B$ is not collinear to any column of $A$, so $A$ is not identifiable.
\end{proof}

\noindent Propositions \ref{prop: sufficient con for identifiability} and~\ref{prop: necessary con for identifiability} combine to prove Theorem~\ref{thm: iff for identifiablity}.
We give an example of a identifiable matrix $A$. We build such examples in Section \ref{section: special matrices}.

\begin{example}
Let 
$$
A=\begin{pmatrix}
0&3&1&\frac{-27417}{160871}&1&0&0&0\\
1&9&11&\frac{282663}{36181}&0&1&0&0\\
2&14&13&17&0&0&1&0\\
3&1&\frac{-89735}{6339}&19&0&0&0&1
\end{pmatrix}.
$$
If $\fb^{\otimes 2} 
\in \Span \{ \fa_1^{\otimes 2}, \ldots, \fa_8^{\otimes 2} \}$, then $\fb$ collinear to $\fa_j$ for some $j \in \{ 1, \ldots, 8 \}$, 
as can be checked in Macaulay2.
So $A$ is identifiable, by Theorem~\ref{thm: iff for identifiablity}.
\end{example}

\section{From identifiability to systems of quadrics}\label{sec:iden to quadrics}

The characterization of identifiability in Theorem~\ref{thm: iff for identifiablity} is closely related to the study of systems of quadrics, as we now describe. 
 We will use systems of quadrics to prove Theorem~\ref{thm: real generic matrix result}. The proof has two steps. The first step is to study the complex analogue of identifiability.
The second step is to convert the complex results into real insights for the real setting of ICA. We give the complex analogue of Definition~\ref{def:id}.

\begin{definition}
\label{def:id_complex_analogue}
A mixing matrix $A \in \CC^{I\times J}$ is \emph{complex identifiable} if for any non-degenerate sources $\mathbf{s}=(s_1,
\ldots,s_J)\T$ with at most one $s_j$ Gaussian, matrix $A$ can be recovered uniquely, up to permutation and scaling of its columns, from $A \mathbf{s}$. That is, if $A \mathbf{s}$ and $ B \mathbf{r}$ have the same distribution for some $B \in \CC^{I \times K}$ with $K \leq J$ and some $\mathbf{r} = (r_1, \ldots, r_K)$, with 
the same number of Gaussian entries as $\mathbf{s}$, then $J = K$ and matrices $A$ and $B$ coincide, up to permutation and scaling of columns. 
\end{definition}

Complex ICA appears in applications to telecommunications~\cite{uddin2015applications} and in ICA algorithms such as~\cite{de2007fourth} and \cite{albera2004blind}.
We prove the complex analogue of Theorem~\ref{thm: iff for identifiablity}.

\begin{proposition}\label{prop: complex identifiable}
A matrix $A\in \CC^{I\times J}$ is \emph{complex identifiable} if and only if no pair of its columns are collinear and the linear span of the matrices $\fa_1^{\otimes 2} , \ldots, ,\fa_J^{\otimes 2}$ does not contain any rank one matrix that is not collinear to $\fa_1^{\otimes 2},\ldots,\fa_J^{\otimes 2}$. 
\end{proposition}

\begin{proof}
    The sufficient direction is the same as the proof of Proposition~\ref{prop: sufficient con for identifiability}. 
    For the necessary direction, the proof is simpler than Proposition~\ref{prop: necessary con for identifiability}, since we allow complex square roots. 
    The matrix $A$ cannot have collinear columns, 
    as in Proposition~\ref{prop: necessary con for identifiability}. 
    Given $\fb^{\otimes 2} \in \Span \{ \fa_j^{\otimes 2} : j = 1, \ldots, J \}$ such that $\fb$ is not collinear to any $\fa_j$, 
    we write 
    $$
    \fb^{\otimes 2} = \sum_{j=1}^{J-1} \lambda_j\fa_j^{\otimes 2}+\fa_J^{\otimes 2},$$
since we can assume without loss of generality that the coefficient of $\fa_J$ is one.
    Define $A$ and $B$ as in the proof of Proposition~\ref{prop: necessary con for identifiability}.
    Let $r_J$ and $s_J$ be standard Gaussians. Choose $J-1$ non-Gaussian random variables $y_1,\ldots,y_{J-1}$ and standard Gaussian distributions $z_1,\ldots,z_{J-1}$.
    Define $r_j=y_j$ and $s_j=y_j+\sqrt{\lambda_j}z_j$. Then $\Psi_{s_j}(t) = \Psi_{y_j}(t) + \lambda_j \Psi_{z_j}(t)$.
    The source variables differ by a complex scalar multiple of a Gaussian. 
    The second characteristic functions for $A\mathbf{s}$ and $B\mathbf{r}$ are equal, by Theorem \ref{thm:secondchar}, since 
    \begin{align}
    \Psi_{A\mathbf{s}}(\mathbf{u})-\Psi_{B\mathbf{r}}(\mathbf{u})=\sum_{j=1}^{J-1} (\Psi_{s_j}-\Psi_{r_j})(\sum_{i=1}^I u_i a_{ij})-\frac{1}{2}(\sum_{i=1}^I u_i a_{iJ})^2+\frac{1}{2}(\sum_{i=1}^I u_ib_i)^2=0.
    \end{align}
    Hence $A\mathbf{s}$ and $B\mathbf{r}$ have the same distribution. The last column of $B$ is not collinear to any column of $A$, so $A$ is not complex identifiable.
\end{proof}

We will prove the following characterization of complex identifiability.

\begin{theorem}\label{thm: complex identifiable}
Let $A \in \CC^{I \times J}$ be generic. Then
\begin{enumerate}
    \item If $J \leq {I \choose 2}$ 
    or if $(I, J) = (2, 2)$ or $(3, 4)$, then $A$ is complex identifiable;
    \item If $J \geq {I \choose 2}+2$ or if $J \geq {I \choose 2} + 1$ for $I\geq 4$, then $A$ is complex non-identifiable.
\end{enumerate}
\end{theorem}

\noindent Theorem~\ref{thm: complex identifiable} immediately implies part 1 of Theorem~\ref{thm: real generic matrix result}.

\begin{corollary}\label{cor: first result from complex to real}
Let $A \in \RR^{I \times J}$ be generic. If $J\leq {I\choose 2}$ or if $(I,J) = (2,2)$ or $(3,4)$, then $A$ is identifiable.
\end{corollary}
\begin{proof}
Such matrices $A$ are complex identifiable, by Theorem~\ref{thm: complex identifiable}. Hence no pair of its columns are collinear and the linear span of $\fa_1^{\otimes 2} , \ldots ,\fa_J^{\otimes 2}$ does not contain any rank one matrix not collinear to $\fa_1^{\otimes 2} , \ldots ,\fa_J^{\otimes 2}$, by Proposition~\ref{prop: complex identifiable}.
In particular, the linear span contains no real rank one matrix. Hence $A$ is identifiable, by Theorem \ref{thm: iff for identifiablity}.
\end{proof}

Theorem~\ref{thm: iff for identifiablity} and Proposition~\ref{prop: complex identifiable} translate to conditions on systems of quadrics, homogeneous degree two polynomials, as we now explain.
Theorem~\ref{thm: iff for identifiablity} involves the linear space $\Span \{ \mathbf{a}_j^{\otimes 2} : j = 1, \ldots, J \}$. 
We view a symmetric matrix $M$ either as an array of $I \times I$ entries $M_{ij}$, for $1 \leq i, j \leq I$, or as a vector of ${I + 1 \choose 2}$ entries $M_{ij}$, for $ 1\leq i \leq j \leq I$.
In our identifiability conditions, two vectors or matrices are equivalent if they agree up to scale, so it is convenient to work in projective space. 
We denote the projectivization of 
$\Span \{ \mathbf{a}_j^{\otimes 2} : j = 1, \ldots, J \}$ by $\Wcal(A)$. 
It is a linear space in $\PP_\CC^{m-1}$, where $m = {I+1\choose 2}$.
The coordinates on $\PP_\CC^{m-1}$ are $\mathbf{z} = ( z_{ij} : 1 \leq i \leq j \leq I)$. 
The space $\Wcal(A)$ is defined by linear relations 
\begin{equation}
    \label{eqn:linear}
    l_1(\mathbf{z}) =\sum \limits_{1\leq i\leq j \leq I}\lambda^{(1)}_{ij}z_{ij} \quad \ldots \quad l_k(\mathbf{z}) =\sum \limits_{1\leq i\leq j \leq I}\lambda^{(k)}_{ij}z_{ij}. 
\end{equation} 
The number of quadrics $k$ is the number of linearly independent conditions that cut out $\Wcal(A)$. In particular, if $\Wcal(A)$ spans the whole space then $k =0$.
We study rank one matrices in $\mathcal{W}(A)$.
The projectivization of the set of rank one $I \times I$ symmetric matrices is the second Veronese embedding of $\PP_\CC^{I-1}$. 
We denote it by $\Vcal_I$. 
It is the image of the map \begin{align*} 
\phi: \PP_\CC^{I-1} & \to \PP_\CC^{m-1} \\ 
[x_1 : \ldots : x_I] & \mapsto [x_1^2 : x_1 x_2: \ldots : x_I^2],
\end{align*} 
see~\cite[Exercise 2.8 and Example 18.13]{harris1992algebraic}.
The intersection $\Vcal_I\cap \Wcal(A)$ consists of all rank one matrices, up to scale, that lie in 
$\Span \{ \mathbf{a}_j^{\otimes 2} : j = 1, \ldots, J \}$. 
In particular, it
contains $\fa_1^{\otimes 2}, \ldots,\fa_J^{\otimes 2}$.
The rank one condition converts~\eqref{eqn:linear} into the system of quadrics
\begin{equation}
    \label{eqn:quadrics}
    f_1(\mathbf{x}) =\sum \limits_{1\leq i\leq j \leq I}\lambda^{(1)}_{ij}x_i x_j \quad \ldots \quad f_k(\mathbf{x}) =\sum \limits_{1\leq i\leq j \leq I}\lambda^{(k)}_{ij}x_i x_j. 
\end{equation} 
The intersection $\Vcal_I\cap \Wcal(A)$ is the vanishing locus of the quadrics $f_1, \ldots, f_k$, which we denote by $V(f_1, \ldots, f_k)$. We say $\{f_1,\ldots,f_k\}$ is a system of quadrics defining $\Vcal_I\cap \Wcal(A)$. 
Proposition~\ref{prop: complex identifiable} says that $A$ is complex identifiable if and only if $V(f_1,\ldots,f_k) = \{ \fa_1^{\otimes 2},\ldots,\fa_J^{\otimes 2} \}$. Theorem~\ref{thm: iff for identifiablity} says that $A$ is identifiable if and only if $V(f_1,\ldots,f_k)$ does not contain any \emph{real} points other than $\fa_1^{\otimes 2},\ldots,\fa_J^{\otimes 2}$.

\begin{example}
\label{ex:A_to_quadric}
Let $A$ be the matrix from Example~\ref{ex:id}. The linear equations defining $\Wcal(A)$ are the rows of the matrix 
$$\begin{pNiceMatrix}[first-row]
z_{11} & z_{12} & z_{13} & z_{14} & z_{22} & z_{23} & z_{24} & z_{33} & z_{34} & z_{44} \\
0 & 1 & 0 & 0 & 0 & -1 & 1 & 0 & 0 & 0\\
0 & 0 & 1 & 0 & 0 & -1 & 1 & 0 & 0 & 0\\
0 & 1 & 0 & 0 & 0 & -1 & 0 & 0 & 1 & 0\\
0 & 0 & 0 & 1 & 0 & 0 & 0 & 0 & 0 & 0\\
\end{pNiceMatrix}.$$
The corresponding system of quadrics defining $\Vcal_I \cap \Wcal(A)$ is obtained by replacing $z_{ij}$ by $x_i x_j$, to give 
    $f_1  = x_1x_2-x_2x_3+x_2x_4$, $f_2  = x_1x_3-x_2x_3+x_2x_4$, $f_3 = x_1x_2-x_2x_3+x_3x_4$, $f_4  = x_1x_4$.
\end{example}

\section{Systems of quadrics}
\label{sec:quadrics}

Quadrics have been studied as far back as 300BC~\cite{heath1921history}. They remain a popular topic in algebraic geometry, see e.g.~\cite{bashelor2008enumerative,odehnal2020universe,fevola2020pencils}.
In this section, we prove results for systems of quadrics, which may be of independent interest, and which are building blocks of our proof of Theorem~\ref{thm: real generic matrix result}. We prove the following quadric restatement of Theorem~\ref{thm: complex identifiable} in Section~\ref{sec:complex_id}.
The case $J\leq {I\choose 2}$ is \cite[Proposition 3.2]{kileel2019subspace}.

\begin{theorem}
\label{thm:quadrics}    
Let $\Vcal_I$ be the second Veronese embedding of $\PP_{\CC}^{I-1}$. Suppose that $\mathbf{v}_1^{\otimes 2},\ldots,\mathbf{v}_J^{\otimes2}$ are generic points on $\Vcal_I$ with $\Wcal(A)$ their projective linear span. Let the system of quadrics defining $\Vcal_I\cap \Wcal(A)$ be $\{f_1,\ldots,f_k\}$, with 
vanishing locus $V(f_1,\ldots,f_k)$.
\begin{enumerate}
    \item If $J\leq {I \choose 2}$ or if $(I,J) = (2,2)$ or $(3,4)$, then $V(f_1,\ldots,f_k) = \{\mathbf{v}_1^{\otimes2},\ldots,\mathbf{v}_J^{\otimes 2}\}.$ 
    \item If $J>{I \choose 2}+1$ or if $J> {I \choose 2}$ for $I\geq 4$, then $V(f_1,\ldots,f_k) \supsetneqq \{\mathbf{v}_1^{\otimes2},\ldots,\mathbf{v}_J^{\otimes2}\}.$ 
\end{enumerate}
\end{theorem}

\begin{example}
We revisit Examples~\ref{ex:id} and~\ref{ex:A_to_quadric}.
A Macaulay2 computation~\cite{grayson2002macaulay2} confirms that $V(f_1,f_2,f_3,f_4)=\{\fa_1^{\otimes 2},\ldots,\fa_6^{\otimes 2}\}$, which proves that $A$ is complex identifiable (and, in particular, identifiable), by Theorem~\ref{thm:quadrics}. 
\end{example}

We prove the following result in Section~\ref{sec:real_quadrics}.
It is used in the identifiability result for $J={I\choose 2}+1$ in Theorem~\ref{thm: real generic matrix result}.
By an open set of systems of $I - 1$ quadrics, we mean the coefficients of the quadrics form an open set in the space of $(I - 1) {I+1\choose 2}$ coefficients.

\begin{theorem}\label{thm: quadrics intersect even integer positive probability}
For every even integer $\ell$ from $0$ to $2^{I-1}$, there is an open set of systems of $I-1$ quadrics 
in $\RR[x_1,\ldots,x_I]$
that have $2^{I-1}$ distinct intersection points, of which $\ell$ are real.
\end{theorem}

\subsection{Complex solutions to a system of quadrics}
\label{sec:complex_id}

In this section, we prove Theorems~\ref{thm:quadrics} and~\ref{thm: complex identifiable}.
As above, let $\Vcal_I$ denote the second Veronese embedding of $\PP_{\CC}^{I-1}$ in $\PP_\CC^{m-1}$, where $m = { I + 1 \choose 2}$.

\begin{lemma}\label{lemma:geenric full rank}
Let $\mathbf{v}_1^{\otimes 2},\ldots,\mathbf{v}_J^{\otimes 2}$ be generic points of $\Vcal_I$. Then the matrix in $\CC^{{{I+1}\choose 2}\times J}$ with columns $\mathbf{v}_1^{\otimes 2},\ldots,\mathbf{v}_J^{\otimes 2}$ has full rank.
\end{lemma}
\begin{proof}
The highest rank is attained for generic matrices, so it suffices to exhibit an example with full rank. 
Suppose $\mathbf{e}_1,\ldots,\mathbf{e}_I$ are canonical basis vectors in $\CC^I$.
Let $S = \{\mathbf{e}_1^{\otimes 2},\ldots,\mathbf{e}_I^{\otimes 2} 
\} \cup \{ (\mathbf{e}_i+\mathbf{e}_j)^{\otimes 2}: i<j\}$.
Then $|S| = m = {{I + 1 \choose 2}}$.
A subset of $S$
of size $J \leq m$ 
is linearly independent. 
When $J> m$, taking the union of $S$ with any $J-m$ symmetric rank one matrices forms a linear space of dimension $m$.
In both cases, the matrix has full rank. 
\end{proof}

\begin{lemma}\label{lemma: generic A.A intersects in 2 power I-1 points}
For ${I\choose 2}+1$ generic points in $\Vcal_I$, with span $\Wcal$, the intersection $\Vcal_I\cap\Wcal$ consists of $2^{I-1}$ distinct points. 
\end{lemma}
\begin{proof}
The variety $\Vcal_I$ has dimension $I-1$ and degree $2^{I-1}$~\cite[Exercise 2.8 and Example 18.13]{harris1992algebraic}. 
A generic linear space of codimension $I-1$ therefore intersects $\Vcal_I$ in $2^{I-1}$ distinct points, by B\'ezout's Theorem. 
Our goal is to show that $\Wcal$ from the statement 
is sufficiently generic:
a codimension $I-1$ subspace that intersects $\mathcal{V}_I$ in degree many distinct points.

Let $J = { I \choose 2} + 1$ and $m = {{I + 1 \choose 2}}$.
The space $\Wcal$ is spanned by $J$ points, and lives in $\mathbb{P}_\CC^{m-1}$. It has projective dimension $(J-1)$, by 
Lemma~\ref{lemma:geenric full rank}. Hence it has codimension $I-1$, since $(I - 1) + (J - 1) = m-1$.
That is, for generic $\mathbf{v}_1, \ldots, \mathbf{v}_J\in \PP_\CC^{I-1}$, the projective linear space $\Span \{ \mathbf{v}_1^{\otimes 2},\ldots,\mathbf{v}_J^{\otimes 2} \}$ is an element of 
the Grassmannian variety of $(J-1)$-dimensional projective linear spaces in $\PP_\CC^{m-1}$, which we denote 
$\Gr(J-1, m-1)$.

Let $\mathcal{U} \subset \Gr(J-1, m-1)$ be the set of spaces spanned by $J$ points on $\Vcal_I$. The set $\mathcal{U}$ is open and dense in $\Gr(J-1, m-1)$, as follows. 
For a generic $(J-1)$-dimensional linear space $\mathcal{L}$, the intersection $\Vcal_I \cap \mathcal{L}$ spans~$\mathcal{L}$, by~\cite[Proposition 18.10]{harris1992algebraic} applied $I-1$ times, since the variety $\Vcal_I$ is irreducible and non-degenerate (not contained in any hyperplane) and its generic hyperplane sections are also non-degenerate and irreducible if $\dim \Vcal_I\geq 2$. 
Choosing a basis of $\mathcal{L}$ from $\Vcal_I \cap \mathcal{L}$ shows that $\mathcal{U}$ is open and dense.

Let $\mathcal{U}^\prime$ denote the elements of $\Gr(J-1, m-1)$ that intersect $\Vcal_I$ in degree many distinct points. 
Then $\mathcal{U}^\prime$ is open and dense, by B\'ezout's theorem.
Hence $\mathcal{U} \cap \mathcal{U}^\prime$ is open and dense in $\Gr(J-1, m-1)$. 
Define the map 
\begin{align*} \Phi: (\mathbb{P}_\CC^{I-1})^J & \dashedrightarrow \Gr(J-1, m-1) \\ 
(\mathbf{v}_1,\ldots,\mathbf{v}_J) & \mapsto \Span \{ \mathbf{v}_1^{\otimes 2},\ldots,\mathbf{v}_J^{\otimes 2} \} ,
\end{align*}
where $\Span$ here denotes the projective span.
It is defined almost everywhere, by Lemma~\ref{lemma:geenric full rank}. 
The pre-image $\Phi^{-1}(\mathcal{U} \cap \mathcal{U}^{\prime})$ consists of collections of points for which $\Vcal_I \cap \Wcal$ is $2^{I-1}$ distinct points. 
As the pre-image of a dense open set, it is dense and open in $(\mathbb{P}_\CC^{I-1})^J$. 
\end{proof}

\noindent We use the following algebraic geometry result to prove Theorem \ref{thm:quadrics}.

\begin{theorem}[Generalized Trisecant Lemma, see {\cite[Proposition 2.6]{chiantini2002weakly}}]\label{lemma:trisecant}
Let $X \subseteq \PP_\CC^{m-1}$ be an irreducible, reduced, non-degenerate projective variety of dimension $I-1$ and let $J$ be a non-negative integer with $(J-1) + (I - 1) < m-1$.
Let $P_1,\ldots,P_J$ be general points on $X$. Then the intersection of $X$ with
the subspace spanned by $P_1,\ldots,P_J$ is the points $P_1,\ldots,P_J$.
\end{theorem}

\begin{proof}[Proof of Theorem~\ref{thm:quadrics}]
Let $m = {I + 1 \choose 2}$. 
Assume that $J>{I\choose 2}+1$. 
For generic $\mathbf{v}_1, \ldots, \mathbf{v}_J\in \PP_\CC^{I-1}$, the space
$\Wcal = \Span \{ \mathbf{v}_1^{\otimes 2}, \ldots, \mathbf{v}_J^{\otimes 2} \}$ has projective codimension at most $I-2$, 
by Lemma \ref{lemma:geenric full rank}.
The dimension of $V(f_1,\ldots,f_k)$ is therefore at least $(I-1)-(I-2)=1$, 
by Krull's Principal Ideal Theorem~\cite[Theorem 1.11A]{hartshorne2013algebraic}. Hence there are infinitely many points in $\Vcal_I \cap \Wcal$, so $V(f_1,\ldots,f_k)\supsetneqq \{\mathbf{v}_1^{\otimes 2},\ldots,\mathbf{v}_J^{\otimes 2}\}$.

Assume $J={I\choose 2}+1$. 
For generic 
$\mathbf{v}_1^{\otimes 2},\ldots,\mathbf{v}_J^{\otimes 2}$, the intersection $\Vcal_I \cap \Wcal$ consists of $2^{I-1}$ distinct points, by Lemma~\ref{lemma: generic A.A intersects in 2 power I-1 points}. 
When $I\geq 4$, we have $2^{I-1}> J$, hence $V(f_1,\ldots,f_k)\supsetneqq \{\mathbf{v}_1^{\otimes 2},\ldots,\mathbf{v}_J^{\otimes 2}\}$. When $I \leq 3$, we have $2^{I-1}= J$, so $V(f_1,\ldots,f_k)=\{\mathbf{v}_1^{\otimes 2},\ldots,\mathbf{v}_J^{\otimes 2}\}$.

It remains to consider $J\leq {I\choose 2}$. The Veronese variety $\Vcal_I\in \PP^{m-1}$ is irreducible, reduced and non-degenerate with dimension $I-1$. We have $(J-1)+(I-1)<m-1$. Hence $V(f_1,\ldots,f_k)=\{\mathbf{v}_1^{\otimes 2},\ldots,\mathbf{v}_J^{\otimes 2}\}$ for generic $\mathbf{v}_1, \ldots, \mathbf{v}_J\in \PP_\CC^{I-1}$,
by Theorem \ref{lemma:trisecant}.
\end{proof}

\begin{proof}[Proof of Theorem~\ref{thm: complex identifiable}]
Let $\mathbf{v}_1,\ldots,\mathbf{v}_J$ in Theorem~\ref{thm:quadrics} be $\fa_1,\ldots,\fa_J$.
 Theorem~\ref{thm: complex identifiable} is equivalent to the statement about quadrics in Theorem~\ref{thm:quadrics}, see the end of Section~\ref{sec:iden to quadrics}. 
\end{proof}

\subsection{Real solutions to a system of quadrics}
\label{sec:real_quadrics}

In this section, we prove Theorem~\ref{thm: quadrics intersect even integer positive probability}.

\begin{proof}[Proof of Theorem~\ref{thm: quadrics intersect even integer positive probability}]
A system of $I-1$ homogeneous quadrics in $\RR[x_1,\ldots,x_I]$ generically has $2^{I-1}$ complex solutions.  
There is a dense open set of quadric systems 
whose solution set consists of $2^{I-1}$ distinct complex points. The number of real solutions is constant on the connected components of this set. Hence it suffices to find one system with $\ell$ distinct real solutions for each even $0 \leq \ell \leq 2^{I-1}$.
There is a dense open set of quadric systems 
such that any solution has $x_I\neq 0$. Without loss of generality, we dehomogenize the quadrics, intersecting them with the plane $x_I=1$. Then it suffices to find $I-1$ inhomogeneous quadrics in $\RR[x_1,\ldots,x_{I-1}]$ that intersect in $2^{I-1}$ distinct points with $\ell$ distinct real solutions for each even $0 \leq \ell \leq 2^{I-1}$.
We prove this by induction.

When $I=2$, we have a single univariate quadric $g(x)=ax^2+bx+c$. It generically has two distinct roots; there are $0$ or $2$ real roots, depending on the sign of $b^2-4ac$.

Assume the result for $I$: there is a system of $I-2$ quadrics in $\RR[x_1,\ldots,x_{I-1}]$ with $2^{I-2}$ distinct solutions and $\ell$ of them real, for all even 
values $0 \leq \ell \leq 2^{I-2}$.
Choose a real value $\alpha$ such that no solution has $x_1 = \alpha$. Then adding the quadric $(x_1 - \alpha)^2 - x_{I-1}^2$ gives $I-1$ quadrics in $\RR[x_1,\ldots,x_{I-1}]$ with $2^{I-1}$ distinct solutions, of which $2 \ell$ are real.
It remains to find a system of quadrics with $2^{I-1}$ distinct solutions, of which $2 \ell - 2$ are real, for every even $\ell$ in the range $2\leq \ell \leq 2^{I-2}$.
Consider our system of $I-2$ quadrics with $2^{I-2}$ distinct solutions, of which $\ell$ are real. We can apply a change of basis to ensure that the $x_1$ coordinates of the roots have distinct values, since the roots are distinct. Choose $\beta \in \RR$ in between the largest and second largest $x_1$ values that appear among the $\ell$ real roots. Then add the quadric $-\beta^2 + x_1^2 +  x_{I-1}^2$. The resulting system has all solutions distinct and $2 \ell - 2$ of them real. 
\end{proof}

\section{From complex to real identifiability}
\label{sec:complex_to_real}

We specialize from complex to real identifiability to prove Theorem \ref{thm: real generic matrix result}.
Results to study the real solutions are in Section \ref{sec:technical results} and the proof of Theorem~\ref{thm: real generic matrix result} is in Section~\ref{sec:generic identifiability}.

\subsection{The projected second Veronese}\label{sec:technical results}

We introduce the projected second Veronese variety and compute its dimension and degree. 
We then give a criterion for the existence of real points in a variety. The proof of Theorem~\ref{thm: real generic matrix result} applies the criterion to the projected second Veronese variety. 

Changing basis on $\RR^I$ does not affect the identifiability of $A \in \RR^{I \times J}$. That is, when $J \geq I$, if $A$ is identifiable, so is $M A$ for all invertible $M\in \RR^{I \times I}$.
We can therefore assume without loss of generality that a generic $A \in \RR^{I \times J}$ has the form
\begin{equation}
    \label{eqn:specialA}
A=\begin{pmatrix}
    \vdots& &\vdots&\vdots& &\vdots \\
    \mathbf{a}_1 &\cdots&\mathbf{a}_{J-I}&\mathbf{e}_1&\cdots&\mathbf{e}_I \\
    \vdots& &\vdots&\vdots& &\vdots 
\end{pmatrix}.
\end{equation}
This motivates the following definition.

\begin{definition}[The projected second Veronese variety]
\label{definition: projected Veronese}
Consider the map $\varphi: \mathbb{C}^{I} \to \mathbb{C}^{{I \choose 2}}$ with $\varphi(x_1,\ldots,x_I)=(x_1x_2,\ldots,x_{I-1}x_I)$ and the projection map $\pi: (\mathbb{C^\star})^{{I \choose 2}}\to \mathbb{P}_\CC^{{I \choose 2}-1}$. The $I$-th projected second Veronese embedding, denoted $\mathcal{Z}_I$, is the closure of $\pi\circ \varphi((\CC^\star)^I)$ in $\mathbb{P}_\CC^{{I \choose 2}-1}$.
\end{definition}

\begin{proposition}
\label{prop:projected_veronese}
    Let $A \in \RR^{I \times J}$ have the form~\eqref{eqn:specialA}. Let $\Wcal(A)_\pi$ be the projective linear space spanned by $\varphi(\mathbf{a}_1),\ldots, \varphi(\mathbf{a}_{J-I})$. Then $A$ is identifiable if and only if no pair of its columns are collinear and the only real points in the intersection
$\Wcal(A)_\pi \cap \Zcal_I$ are $\varphi(\mathbf{a}_1),\ldots,\varphi(\mathbf{a}_{J-I})$.
\end{proposition}

\begin{proof}
The matrix $A$ is identifiable if and only if no pair of its columns are collinear and the real points in the intersection $\Wcal(A)\cap \Vcal_I$ are $\fa_1^{\otimes 2},\ldots,\fa_{J-I}^{\otimes 2},\mathbf{e}_1^{\otimes 2},\ldots, \mathbf{e}_I^{\otimes 2}$, by Theorem~\ref{thm: iff for identifiablity}.
The span of $\mathbf{e}_1^{\otimes 2},\ldots, \mathbf{e}_I^{\otimes 2}$ is the diagonal matrices. Hence, $\fb^{\otimes 2}$ lies in $\Wcal(A)\cap \Vcal_I$ if and only if its off-diagonal part $\varphi(\fb)$ lies in $\Wcal(A)_\pi \cap \Zcal_I$.
\end{proof}

\begin{lemma}\label{lemma:degree projective veronese}
The projected second Veronese variety $\Zcal_I$ is a toric variety of dimension $I-1$ and degree $2^{I-1}-I$.
\end{lemma}
\begin{proof}
We use the Hilbert polynomial to compute the dimension and degree of $\Zcal_I$, see \cite[Section 1.7]{hartshorne2013algebraic}.
Let $h(\ell)$ be the dimension of degree $\ell$ polynomials in the coordinate ring $\CC[\Zcal_I]$. 
These are degree $2\ell$ polynomials obtained from products of $x_1 x_2$, \ldots, $x_{I-1} x_I$.
Thus, if $x_i^{\ell + 1}$ divides a monomial, it cannot be degree $2\ell$ in $\CC[\Zcal_I]$. 
A monomial $x_1^{a_1}x_2^{a_2}\cdots x_I^{a_I}$ is in $\CC[\Zcal_I]$ if and only if $a_1+\ldots+a_I=2\ell$ with $1\leq a_i \leq \ell$ for all $i$. Hence
$$ h(\ell)={2\ell+I-1\choose I-1}-I{I+\ell-1\choose I-1},
$$ 
a polynomial in $\ell$ with leading term $\frac{2^{I-1}-I}{(I-1)!}\ell^{I-1}$. Hence $\dim \Zcal_I = I-1$ and $\deg \Zcal_I = 2^{I-1}-I$.
\end{proof}

\subsection{Generic identifiability}\label{sec:generic identifiability}
In this section, we prove Theorem \ref{thm: real generic matrix result}.

\begin{lemma}\label{lemma: J good then Jprime good}
If a generic matrix in $\RR^{I \times J}$ is non-identifiable, then a generic matrix in $\RR^{I \times J^\prime}$ is non-identifiable for all $J^\prime > J$.
\end{lemma}
\begin{proof}
Fix a generic matrix $A \in \RR^{I \times J^\prime}$. The submatrix consisting of the first $J$ columns of $A$ is a generic $I \times J$ matrix, hence is non-identifiable by assumption. 
So, the intersection $\Span \{ \mathbf{a}_j^{\otimes 2} : 1 \leq j \leq J \} \cap \Vcal_I$ contains a real point that is not collinear to any of $\{\mathbf{a}_j^{\otimes 2}: 1\leq j \leq J \}$, by Theorem~\ref{thm: iff for identifiablity}. This point is not collinear to any column of $A$, by genericity.
\end{proof}

\begin{proposition}\label{prop: I=0,1 mod 4, J=Ichoose 2+1 not iden}
For $I\equiv 0,1 \mod 4$ and $J={I \choose 2}+1$, a generic $A\in \mathbb{R}^{I\times J}$ is non-identifiable.
\end{proposition}
\begin{proof}
The intersection $\mathcal{V}_I \cap \Wcal(A)$ consists of $2^{I-1}$ distinct points, by Lemma \ref{lemma: generic A.A intersects in 2 power I-1 points}. 
Complex intersection points come in pairs, since $\Vcal_I\cap \Wcal(A)$ is the vanishing locus of quadrics with real coefficients. 
So there is an even number of real points in $\Vcal_I\cap \Wcal(A)$. 
There are $J$ real points, which correspond to the columns of $A$. If $I\equiv 0,1 \mod 4$, then $J = {I \choose 2}+1$ is odd. Hence there is an extra real solution, so $A$ is not identifiable, by Theorem \ref{thm: iff for identifiablity}.
\end{proof}

\noindent Combining the above with Lemma \ref{lemma: J good then Jprime good} gives the following.

\begin{corollary}\label{cor: I equals 0,1 mod 4 non-identifiable}
When $I\equiv 0,1 \mod 4$ and $J\geq{I \choose 2}+1$, a generic $A\in \mathbb{R}^{I\times J}$ is non-identifiable.
\end{corollary}

Recall the map from Definition~\ref{definition: projected Veronese}, with $\varphi(x_1,\ldots,x_I)=(x_1x_2,\ldots,x_{I-1}x_I)$. 
We study generic $A \in \RR^{I \times J}$ via the images under $\varphi$ of $J - I$ generic vectors, by Proposition~\ref{prop:projected_veronese}.

\begin{lemma}\label{lem: odd degree positive dimension infinitely many real}
Let $\ideal$ be a homogeneous ideal generated by polynomials with real coefficients and $X$ the vanishing locus of $\ideal$. 
If $X$ has odd degree, then it contains a real point. If moreover $\dim X\geq 1$, then $X$ contains infinitely many real points.
\end{lemma}
\begin{proof}
The ideal $\ideal$ is generated by polynomials with real coefficients, so complex solutions come in pairs. 
If $\dim X=0$, then $X$ contains a real point. If $d = \dim X \geq 1$, a generic real linear space of codimension $d$ intersects $X$ to give an odd number of points, so $X$ contains a real point. 
Assume for contradiction that $X$ contains only finitely many real points. There is a generic real codimension $d$ linear space that does not pass through these points, but that intersects $X$ in degree many points. This intersection contributes a new real point, a contradiction.
\end{proof}

\begin{proposition}\label{prop: I=3mod4 J>Ichoose2+1}
If $I\equiv 2,3 \mod 4$ and $J>{I \choose 2}+1$, a generic $A\in \mathbb{R}^{I\times J}$ is non-identifiable.
\end{proposition}

\begin{proof}
Let $X:= \Wcal(A)_\pi \cap \Zcal_I$. The matrix of first $J-I$ columns of $A$ is generic, so $\Wcal(A)_\pi$ has projective dimension $J-I-1$. Hence $\dim X = (I-1)+(J-I-1)-({I\choose 2}-1)>0$ and $\deg X = 2^{I-1}-I$, 
by similar arguments as Lemma \ref{lemma: generic A.A intersects in 2 power I-1 points}.

When $I \equiv 3 \mod 4$, the degree of $X$ is odd. Hence $X$ contains infinitely many real points, by Lemma \ref{lem: odd degree positive dimension infinitely many real}.
It remains to consider $I \equiv 2 \mod 4$. If $J={I\choose 2}+2$,
we consider the system of quadrics in $I-1$ variables obtained by setting $x_I = 0$. Denote the quadrics by $g_1, \ldots, g_\ell$, where $\ell$ is the codimension of $\Wcal(A)_\pi$. 
Its vanishing locus consists of $2^{I-2}-(I-1)$ points, by Lemma \ref{lemma:degree projective veronese} and similar arguments to Lemma \ref{lemma: generic A.A intersects in 2 power I-1 points}. 
Since $2^{I-2}-(I-1)$ is odd, 
there is a real point $\varphi(y_1, \ldots, y_{I-1})$ in the intersection,
by Lemma \ref{lem: odd degree positive dimension infinitely many real}. 
Hence 
$\varphi(y_1 , \ldots, y_{I-1} , 0) \in \Wcal(A)_\pi\cap \Zcal_I$. This point is not collinear to any column of $A$, by genericity. The case $J>{I\choose 2}+2$ follows from Lemma \ref{lemma: J good then Jprime good}.
\end{proof}

The cases remaining are $I\equiv 2,3 \mod 4$, $I\geq 4$ and $J={I\choose 2}+1$.
The following result follows from Theorem \ref{thm: complex identifiable}. We will use it to prove these remaining cases.

\begin{corollary}\label{cor: generic linear space and generic A}
Let $m = {{I + 1 \choose 2}}$. For a generic linear space $\Wcal\subseteq \PP_{\mathbb{C}}^{m-1}$ of dimension $I\choose 2$,
any ${I\choose 2}+1$ points in the intersection $\Wcal \cap \Vcal_I$ are linearly independent as affine vectors.
\end{corollary}
\begin{proof}
Fix a generic linear space $\Wcal$ of dimension $I\choose 2$. It intersects $\Vcal_I$ in $2^{I-1}$ distinct points. The intersection points span $\Wcal$, 
by~\cite[Proposition 18.10]{harris1992algebraic} applied $I-1$ times.

Assume for contradiction that there is a set $S$ of ${I\choose 2}+1$ points in $\Wcal \cap \Vcal_I$ that are linearly dependent. 
They span a linear space $\Wcal^\prime$ of dimension $\ell <{I\choose 2} + 1$. 
Choose a subset of $S$ of size $\ell$ that spans $\Wcal^\prime$.
These points define a matrix $A \in \CC^{I \times \ell}$ that is not complex identifiable, by Proposition~\ref{prop: complex identifiable}.

Let $S_{\ell}$ be the sets of $\ell$ linearly independent vectors in $\PP_\CC^{I-1}$ whose corresponding matrices in $\CC^{I \times \ell}$ are complex non-identifiable. 
Complex identifiability holds generically, by Theorem \ref{thm: complex identifiable}, since $\ell<{I\choose 2}+1$.
Hence $\dim S_{\ell} < \dim ((\PP_\CC^{I-1})^{\ell})=\ell(I-1)$.

Define the continuous map $\phi_{\ell}$ that sends a collection of ${I\choose 2}+1$ vectors, the first $\ell$ of which are in $S_{\ell}$, to the linear space their second outer products span. 
Then $\Wcal$ is in the image of $\phi_{\ell}$, 
since it is spanned by the $\ell$ points spanning $\Wcal'$ plus ${I\choose 2}+1-\ell$ other points. 
The dimension of $\operatorname{im} \phi_{\ell}$ is at most $\dim S_{\ell}+(I-1)({I\choose 2}+1-\ell)<(I-1)({I\choose 2}+1)$. But the space of ${I \choose 2}$ dimensional spaces in $\PP_\CC^{m-1}$ has dimension $(I-1)({I\choose 2}+1)$, by~\cite[Lecture 6]{harris1992algebraic}, a contradiction.
\end{proof}

\noindent Corollary \ref{cor: generic linear space and generic A} is still true for a generic real linear space, since a generic real linear space of dimension $n$ is a generic complex linear space of dimension $n$.

\begin{proposition}\label{prop: id and non id positive probability}
Let $I\equiv 2,3 \mod 4$, $I\geq 4$ and $J={I\choose 2}+1$. 
For matrices in $\mathbb{R}^{I\times J}$,
identifiability and non-identifiability both occur with positive probability.
\end{proposition}

\begin{proof}[Proof of of Proposition \ref{prop: id and non id positive probability}]
We construct non-empty open sets of identifiable and non-identifiable matrices.
More specifically, we find open sets $U_1$ and $U_2$ in $\RR^{I \times J}$ such that for each matrix in $U_1$, the corresponding system of quadrics has $J$ real solutions, and the system of quadrics for $U_2$ has $J+2$ real solutions.
To find $U_1$ and $U_2$, we construct a continuous map from matrices to quadric systems and use Theorem \ref{thm: quadrics intersect even integer positive probability}.

We construct a map that sends a matrix $A \in \RR^{I \times J}$ to the system of quadrics that define $\Wcal(A) \cap \Vcal_I$. 
The map is continuous on the dense set of full rank matrices $A$ with $\dim \Wcal(A) = J-1$ and such that $\Wcal(A)$ intersects $\Vcal_I$ generically transversely. 
There is a continuous function that sends a projective $(J-1)$-dimensional linear space to a choice of $I-1$ linear relations defining it, e.g. using the orthogonal complement. We compose it with the map that sends the linear relation $\sum \lambda_{ij} z_{ij}$ to the quadric $\sum \lambda_{ij} x_i x_j$.
Finally, we pre-compose it with the continuous map that sends 
$A \in \RR^{I \times J}$ to $\Wcal(A)$.
Call the resulting map $\psi$.

There exist open sets $\Ucal_1$ (respectively $\Ucal_2$) in $(\PP_\RR^{m-1})^{I-1}$ such that the $I-1$ quadrics intersect in $2^{I-1}$ distinct points with $J$ (respectively $J+2$) of them real, by Theorem \ref{thm: quadrics intersect even integer positive probability}.
Among these real solutions, $J$ will generically be linearly independent, by Corollary \ref{cor: generic linear space and generic A}. 
Hence there exists $A \in \RR^{I \times J}$ in the preimage of $\psi$.
Set $U_i = \psi^{-1}(\Ucal_i)$ for $i = 1, 2$.
\end{proof}

\begin{proof}[Proof of Theorem \ref{thm: real generic matrix result}]
Corollary \ref{cor: first result from complex to real} gives the first part. 
Proposition \ref{prop: id and non id positive probability} gives the second part.
The third part follows from Corollary \ref{cor: I equals 0,1 mod 4 non-identifiable} and Proposition \ref{prop: I=3mod4 J>Ichoose2+1}.
\end{proof}

\section{Identifiable and non-identifiable matrices}\label{section: special matrices}

In this section we study special matrices: non-identifiable matrices in the range of $(I,J)$ where identifiability generically holds, and 
identifiable matrices in the range of $(I,J)$ where non-identifiability generically holds. We focus mostly on complex identifiablity. 
It is an open problem to find real analogues of some of the results.
We comment throughout on implications for (real) identifiability. 
Recall that $A\in \CC^{I\times J}$, for $I \geq 4$, is generically complex identifiable if $J\leq {I\choose 2}$ and generically complex non-identifiable if  $J \geq {I\choose 2} + 1$, see Theorem~\ref{thm: complex identifiable}. We study large identifiable matrices in Section~\ref{sec:large}, low-rank identifiable matrices in Section~\ref{sec:lowrank}, and non-identifiable special matrices in Section~\ref{sec:nonid_special}.

\subsection{Large identifiable matrices}
\label{sec:large}

In this section we prove the following.

\begin{theorem}\label{thm: when compelx identifiable exists for J>Ichoose 2}
There exist complex identifiable matrices of size $I\times J$ if and only if $J\leq 2^{I-1}$.
\end{theorem}

\begin{proof}
If $\dim (\Wcal\cap \Vcal_I)$ is finite, then the number of intersection points is at most the degree $2^{I-1}$. Hence, if there are at least $J>2^{I-1}$ points in the intersection, we have $\dim (\Wcal\cap \Vcal_I)\geq 1$ and the matrix is complex non-identifiable.

It remains to consider $J\leq2^{I-1}$. 
For every $1\leq k\leq 2^{I-1}$, we show that there exists a projective linear space $\Wcal$ of projective dimension ${I\choose 2}$ such that $\Wcal\cap \Vcal_I$ consists of exactly $k$ points (counted without multiplicity), as follows.

The intersection $\Wcal\cap \Vcal_I$ is the vanishing locus of $I-1$ homogeneous quadrics in $I$ variables.
When $I=2$, the statement is the fact that a quadratic equation can have 1 or 2 complex roots.
When $I=3$, the result follows from~\cite[Example 3]{fevola2020pencils}.

We use induction. Assume we can construct a system of $I-2$ quadrics $f_1,\ldots,f_{I-2}$ in $I-1$ variables $x_1,\ldots,x_{I-1}$ such that the vanishing locus has dimension 0 and there are $k$ points in the vanishing locus of $f_1,\ldots,f_{I-2}$, for some $1\leq k\leq 2^{I-2}$. After a change of basis, we can ensure all points in the vanishing locus have first coordinate nonzero. By adding the quadric $x_1^2-x_I^2$, we obtain $I-1$ quadrics in $I$ variables with $2k$ intersection points (counted without multiplicity). 
For odd values, we apply a change of basis such that one point in the vanishing locus of $f_1,\ldots,f_{I-2}$ has first coordinate $0$. Adding the quadric $x_1^2-x_I^2$ gives a system of $I-1$ quadrics in $I$ variables that intersect in $2k-1$ points (counted without multiplicity).
\end{proof}

There exist (real) identifiable matrices $A \in \RR^{I \times J}$ whenever $J\leq 2^{I-1}$, by similar arguments. But our arguments do not rule out the existence of an identifiable matrix for $J>2^{I-1}$.

\subsection{Low-rank identifiable matrices}
\label{sec:lowrank}

Complex identifiability occurs generically for $J\leq {I\choose 2}$, by Theorem \ref{thm: complex identifiable}.
Later we study whether the set of complex non-identifiable matrices $A \in \CC^{I \times J}$ is closed, for $J\leq {I\choose 2}$. We will see that the 
obstacle
is the existence of complex identifiable matrices with $\mathbf{a}_1^{\otimes 2}, \ldots, \mathbf{a}_J^{\otimes 2}$ linearly dependent. Here we study such matrices.

\begin{definition}
Denote the columns of $A\in \mathbb{R}^{I_1\times J}$ and $B\in \mathbb{R}^{I_2\times J}$ by $\mathbf{a}_1,\ldots,\mathbf{a}_J$ and $\mathbf{b}_1,\ldots,\mathbf{b}_J$, respectively. The \emph{Khatri-Rao product} is the matrix $A\odot B \in \mathbb{R}^{I_1I_2\times J}$ with $j$th column $\mathbf{a}_j\otimes\mathbf{b}_j$, vectorized into a vector of length $I_1 I_2$. We consider $A\odot A \in \RR^{I^2  \times J}$ as a matrix of size ${I+1\choose 2}\times J$, by deleting the repeated rows. 
\end{definition}

\noindent The projectivization of the column space of $A\odot A$ is the projective linear space $\Wcal(A)$.

\begin{proposition}\label{prop: identifiable but khatri-rao not full rank}
There exist complex identifiable $A \in \CC^{I \times J}$, where $J\leq {I\choose 2}$ 
and $\rank (A\odot A) < J$, if and only if $J\geq 8$.
\end{proposition}
\begin{proof}
There exists a complex identifiable matrix $A \in \CC^{I\times J}$ when $J \leq 2^{I-1}$, by Theorem~\ref{thm: when compelx identifiable exists for J>Ichoose 2}. Such a matrix has $\rank (A\odot A) < J$ if $J \geq {{I \choose 2}} + 2$. 
Hence if there exists $I^\prime$ with ${I^\prime\choose 2}+2\leq J \leq 2^{I^\prime-1}$, then we can construct complex identifiable $A^\prime \in \CC^{I^\prime \times J}$ with $\rank(A^\prime \odot A^\prime) < J$. Taking an $I^\prime$-dimensional subspace in $\CC^I$ gives a complex identifiable matrix $A \in \CC^{I \times J}$ with $\rank(A\odot A) < J$. For $J\geq 12$ or $J=8$, there exist such values $I^\prime$.

It remains to consider $9 \leq J \leq 11$. Since $J\leq {I\choose 2}$, we have $I\geq 5$. We give examples of complex identifiable matrices with $\rank(A\odot A) < J$ for $(I,J)=(5,9),(5,10),(5,11)$. For larger $I$, we take their first five rows and set the remaining rows to zero.
Let
$$
A_1 = \begin{pmatrix}
0&3&1&\frac{-27417}{160871}\\
1&9&11&\frac{282663}{36181}\\
2&14&13&17\\
3&1&\frac{-89735}{6339}&19\\
0&0&0&0
\end{pmatrix}, \qquad 
\mathbf{a}_5 = \begin{pmatrix}
1 \\ 1 \\ 1 \\ 1 \\ 1 
\end{pmatrix},
\qquad 
\mathbf{a}_6 = \begin{pmatrix}
1 \\ 2 \\ 3 \\ 4 \\ 7 
\end{pmatrix}.
$$
Let $I_5$ denote the $5 \times 5$ identity matrix.
We check that matrices $\begin{pmatrix} A_1 & I_5 \end{pmatrix} \in \CC^{5 \times 9}, \begin{pmatrix}
    A_2 & \mathbf{a}_5 & I_5
\end{pmatrix}\in \CC^{5 \times 10}, \begin{pmatrix}
    A_3& \mathbf{a}_5 & \mathbf{a}_6 & I_5
\end{pmatrix}\in \CC^{5 \times 11}$ are complex identifiable and have $\rank(A\odot A)< J$.

Conversely, we show that there is no complex identifiable $A\in\CC^{I\times J}$ with $\rank (A\odot A)<J$, 
when $J\leq {I\choose 2}$
and $J\leq 7$.
Such matrices have $3\leq\rank A \leq J-3$, as follows. 
If $\rank A\leq 2$, then $A$ either has two collinear columns or, after a change of basis, it has columns $\mathbf{e}_1,\mathbf{e}_2,\mathbf{e}_1+\mathbf{e}_2$. This is not complex identifiable, since $(\mathbf{e}_1+a\mathbf{e}_2)^{\otimes 2}$ is a linear combination of $\mathbf{e}_1^{\otimes 2},\mathbf{e}_2^{\otimes 2},(\mathbf{e}_1+\mathbf{e}_2)^{\otimes 2}$ for any $a\in \CC$. By the same argument, any matrix with three linearly dependent columns is not complex-identifiable. Hence $\rank A \geq 3$. If $\rank A\geq J-2$, then after a change of basis the first $J-2$ columns of $A$ are $\mathbf{e}_1,\ldots,\mathbf{e}_{J-2}$, and either $A$ has collinear columns or $A\odot A$ has full column rank.
This bound on $\rank A$ and $J \leq {{I \choose 2}}$ implies that there are no complex identifiable examples for $J \leq 5$.

When $J=6$, we need $I\geq 4$ since $J\leq {I\choose 2}$ and we have $3\leq \rank A\leq 6-3=3$. After a change of basis, we obtain a $3\times 6$ complex identifiable matrix, which is impossible by Theorem \ref{thm: when compelx identifiable exists for J>Ichoose 2}.
When $J=7$, we need $I\geq 4$ since $J\leq {I\choose 2}$ and $3\leq \rank A\leq 7-3=4$. If $\rank A=3$, after a change of basis, we obtain a $3\times 7$ complex identifiable matrix, again impossible by Theorem \ref{thm: when compelx identifiable exists for J>Ichoose 2}. If $\rank A=4$, it can be checked in Macaulay2 that $A\odot A$ does not have full column rank only if $A$ contains 3 linearly dependent columns. 
\end{proof}

One direction of Proposition~\ref{prop: identifiable but khatri-rao not full rank} still holds for real identifiability: all of our matrices are real, and complex identifiability implies real identifiability. The converse is more difficult: we would need the non-existence of identifiable matrices when $J>2^{I-1}$.

\subsection{Non-identifiable matrices}
\label{sec:nonid_special}

In this section, we study non-identifiable and complex non-identifiable matrices.

\begin{proposition}\label{prop: exist non-identifiable for small J}
There exist non-identifiable matrices $A \in \RR^{I\times J}$ with no pair of collinear columns for any $3 \leq J\leq {I\choose 2}$. 
\end{proposition}
\begin{proof}
Let $J={I\choose 2}$. Take $A$ to have first three columns $\mathbf{e}_1,\mathbf{e}_2, \mathbf{e}_1 + \mathbf{e}_2$ and remaining columns $\mathbf{e}_i+\mathbf{e}_j$ where $i< j$ and $(i,j)\neq (1, 2), (1,3), (2,3)$. This matrix has no collinear columns.
It is not identifiable, as $(\mathbf{e}_1+a\mathbf{e}_2)^{\otimes 2}$ is a linear combination of $\mathbf{e}_1^{\otimes2},\mathbf{e}_2^{\otimes 2},(\mathbf{e}_1+\mathbf{e}_2)^{\otimes 2}$ for any $a$.
When $J<{I\choose 2}$, take the first $J$ columns of $A$. The submatrix is non-identifiable for $J \geq 3$.
\end{proof}

\begin{theorem}\label{thm: S Itimes J not open}
If the set of identifiable matrices of size $I\times J$ is non-empty, it is not closed.
\end{theorem}
\begin{proof}
Assume $\rank A\leq 2$. If $A$ has at least three columns, it either has collinear columns or after a change of basis it has three columns $\mathbf{e}_1,\mathbf{e}_2,\mathbf{e}_1+\mathbf{e}_2$. Then 
$A$ is non-identifiable, 
as in the proof of Proposition \ref{prop: exist non-identifiable for small J}.

If $\rank A\geq 3$, we can assume without loss of generality that the first three columns are $\mathbf{e}_1+\mathbf{e}_2+\mathbf{e}_3,\mathbf{e}_1,\mathbf{e}_2$. We define a sequence of matrices $A_t$, where $A_t$ has third row of $A$ scaled by $t$. In particular, the first column of $A_t$ is $\mathbf{e}_1+\mathbf{e}_2+t\mathbf{e}_3$. If $t\neq 0$, $A_t$ is identifiable. The limit $A_0=\operatorname{lim}_{t\to 0}A_t$ has first three columns: $\mathbf{e}_1+\mathbf{e}_2,\mathbf{e}_1,\mathbf{e}_2$, so it is non-identifiable.
\end{proof}

\noindent We give a test for a point in $\Wcal(A)$ to lie in $\Wcal(A) \cap \Vcal_I$, which is efficient to test in practice.

\begin{lemma}\label{lemma:matrix D}
Fix $A \in \CC^{I \times J}$ with columns $\mathbf{a}_1, \ldots, \mathbf{a}_J$. Let $n={I\choose 2}$. 
Define $C(A) \in \CC^{ n\times {J\choose 2}}$ to be the matrix of $2\times 2$ minors of $A$. Define $D(A)=C(A)\odot C(A)$, a matrix of size ${n+1\choose 2}\times {J\choose 2}$. Then $\sum_{j=1}^J\lambda_j \fa_j^{\otimes 2}\in \Vcal_I$ if and only if $(\lambda_1,\ldots,\lambda_J)$ satisfies 
\begin{equation}\label{eq:kerD}
D(A) (\lambda_1\lambda_2,\ldots,\lambda_{J-1}\lambda_J)=0.
\end{equation}
\end{lemma}
\begin{proof}
Denote the matrix
$\sum_{j=1}^J\lambda_j \fa_j^{\otimes 2}$ by $G$, with entries
$g_{k \ell} = \sum_{j=1}^J \lambda_j a_{k j}a_{\ell j}$.
Suppose $G \in \Vcal_I$. 
The second Veronese embedding $\mathcal{V}_I$ is generated by $2 \times 2$ minors. 
Evaluated on $G$, these are
$$ g_{k \ell }g_{k^\prime \ell ^\prime}-g_{k \ell^\prime}g_{k^\prime \ell}=\sum_{1\leq i<j\leq J} (a_{ki}a_{k^\prime j}-a_{k^\prime i }a_{k j})(a_{\ell i}a_{\ell^\prime j}-a_{\ell^\prime i }a_{\ell j}) \lambda_i \lambda_j,$$
The product
$(a_{ki}a_{k^\prime j}-a_{k^\prime i }a_{k j})(a_{\ell i}a_{\ell^\prime j}-a_{\ell^\prime i }a_{\ell j})$ is the entry of $D(A)$ at 
row $((k,k^\prime),(\ell,\ell^\prime))$ and column $(i,j)$.
Hence the condition is $D(A)(\lambda_1\lambda_2,\ldots,\lambda_{J-1}\lambda_J)=0$.
\end{proof}

When $A\odot A$ has full column rank, $A$ is complex identifiable if and only if $\ker (D(A))\cap \Zcal_J=\emptyset$ where $\Zcal_J$ is the $J$-th projected second Veronese variety. This is faster than checking all $2 \times 2$ minors, especially when $\ker (D(A))$ has small dimension.

Next we study the set of complex non-identifiable matrices $A \in \CC^{I\times J}$, which we denote by $S_{I\times J}$.
We study the closure of $S_{I\times J}$ and give conditions for when $S_{I\times J}$ is closed.
It is an open problem to extend these results to real identifiability. 
The argument in Theorem~\ref{thm: S Itimes J not open} applies to complex identifiability: it shows that $S_{I \times J}$ is not open.

\begin{proposition}\label{prop: Zariski closure}
Define $X_{I\times J} \subset \CC^{I \times J}$ to be the set
$$\{A\in \CC^{I\times J}: \ker (D(A))\cap \Zcal_J=\emptyset \}.$$
For $J\leq {I\choose 2}$, the \emph{closure} of the set of complex non-identifiable matrices is $X_{I\times J}$.
\end{proposition}
\begin{proof}
All complex non-identifiable matrices $S_{I\times J}$ are contained in $X_{I\times J}$, by Lemma \ref{lemma:matrix D}.
Matrices $D \in \CC^{{n + 1 \choose 2} \times {J \choose 2}}$ such that $\ker D \cap \Vcal_J \neq \emptyset$ lie in some closed algebraic variety,
by the Main Elimination Theorem, see e.g. {\cite[ Definition 1, Theorem 1, Chapter 9]{mumford1999red}}.
The entries of $D(A)$ are polynomials in the entries of $A$. 
So $X_{I\times J} \subset \CC^{I \times J}$ is a closed algebraic variety.
The set $S_{I\times J}$ is the projection to $A \in \CC^{I \times J}$ of the set $$\{(\fa_1,\ldots,\fa_J,\fb): \fb^{\otimes 2} \in  \Wcal(A)\cap \Vcal_I, \, \fb \text{ not collinear to any of }\fa_1,\ldots,\fa_J\}.$$
It can be defined by the vanishing and non-vanishing of polynomials, so it is constructible. 
Hence its Zariski closure is the same as its (Euclidean) closure, see~\cite[Exercise 2.3.18, 2.3.19]{hartshorne2013algebraic}.
Define
$Y_{I\times J}=\{A\in \CC^{I\times J}: \rank A\odot A<J\}$.
We have $X_{I\times J}\cap Y_{I\times J}^c \subset S_{I \times J} \subset X_{I \times J}$.
Taking the closures of the three sets in this chain of containments gives the result.
\end{proof}

\begin{proposition}
When $J\leq {I\choose 2}$, the set $S_{I\times J}$ is closed if and only if $J\leq 7$.
\end{proposition}

\begin{proof}
To show that $S_{I\times J}$ is not closed, we need some complex identifiable matrix $A \in \CC^{I \times J}$ such that $A\odot A$ does not have full column rank. Conversely, to show that $S_{I\times J}$ is closed, we need to prove that there is no such matrix. 
Both follow from Proposition \ref{prop: identifiable but khatri-rao not full rank}.
\end{proof}

\section{Numerical experiments}\label{section: numerical result}
We evaluate the performance of Algorithm \ref{alg:recover A} on synthetic and real data. 
The code for our computations can be found at \url{https://github.com/QWE123665/overcomplete_ICA}. 

The second and fourth cumulant tensors $\kappa_2,\kappa_4$ are the input to Algorithm~\ref{alg:recover A}. 
For synthetic data, these are either true population cumulants or sample cumulants. For real data, the tensors are obtained from samples. The first step of Algorithm~\ref{alg:recover A} computes the symmetric tensor decomposition of the fourth cumulant $\kappa_4$, using~\cite[Algorithm 1]{kileel2019subspace}. 
The outputs are unit vectors $\fa_1,\ldots,\fa_{J-1}$. For the second step of Algorithm~\ref{alg:recover A}, we minimize
$$ \min_{\mathbf{v} \in \RR^I, l \in \RR^J} \lVert \kappa_2-\sum_{j=1}^{J-1} l_j\fa_{j}^{\otimes 2}-l_J\mathbf{v}^{\otimes 2} \rVert ,$$
 using Powell's method~\cite{powell1964efficient}.
We initialize at a random unit vector $\mathbf{v} \in \RR^I$ and a random vector $l \in \RR^J$. We normalize the output $\mathbf{v}$ and set it to be the last column $\fa_J$.

We usually use $1000(I+J)$ iterations for the minimization with Powell's method, the default in the python function \verb|scipy.optimize.minimize|. For synthetic datasets on small sample size, and for real data, we increase the number of iterations and run the minimization 10 times and select the best solution.
That is, from 10 outputs $(\mathbf{v}_1,l_1),\ldots,(\mathbf{v}_{10},l_{10})$, we choose the one with the smallest value of $\lVert \kappa_2-\sum_{j=1}^{J-1} (l_i)_j\fa_{j}^{\otimes 2}-(l_i)_J\mathbf{v}_i^{\otimes 2} \rVert$.

\subsection{Synthetic data}

We take as input a matrix $A \in \RR^{I \times J}$ with its columns rescaled to unit vectors, for various $I$ and $J$. Assume that the first $J-1$ columns correspond to the non-Gaussian sources, and that the last column corresponds to the Gaussian source. We compute the cumulants in one of two ways:
\begin{enumerate}
    \item Use the population cumulants, $\kappa_2 = \sum_{j=1}^J \sigma_j \mathbf{a}_j^{\otimes 2}$ and $\kappa_4 = \sum_{j=1}^{J -1} \lambda_j \mathbf{a}_j^{\otimes 4}$, where $\sigma_j$ is the variance of source $j$ and $\lambda_j$ is its fourth cumulant.
    \item Fix sources $\mathbf{s}$ and compute cumulants from samples of $A\mathbf{s}$.
\end{enumerate}
The output of Algorithm~\ref{alg:recover A} is a matrix $A^\prime \in \RR^{I \times J}$ with unit vector columns.
The last column corresponds to the Gaussian source. 

We measure the proximity of $A$ and $A^\prime$. Since identifiability is only up to permutation and rescaling, we allow for re-ordering of the first $J-1$ columns. Rather than searching over all ways to match the first $J-1$ columns of $A$ to those of $A^\prime$, we use a greedy algorithm to approximate the matching, as follows. We fix the first column of $A$, denoted $\fa_1$. We choose one of the first $J-1$ columns of $A^\prime$ whose cosine similarity with $\fa_1$ has largest absolute value. We set this to be the first column of $A^\prime$ (changing its sign if the cosine similarity is negative). Then we select among the remaining $J-2$ columns, the one with the largest absolute cosine similarity with $\fa_2$ and set this as the second column of $A^\prime$ (again, changing the sign if the cosine similarity is negative). We continue until we reach the last column.
Then we compute the relative Frobenius error $$\sqrt{\sum_{i=1}^I \sum_{j=1}^J (a_{ij}-a^\prime_{ij})^2/J}.$$
We study a range of $I$ and $J$, using the population cumulants in~\autoref{fig:accurate tensor 6,7,8,9}. We examine how the error changes with the variance of the Gaussian source in~\autoref{fig:accurate tensor different variances}. 
We test how our algorithm performs with sample cumulant tensors in~\autoref{fig:exponential source different samplesizes}.

\begin{figure}[htbp]
    \centering
\begin{tikzpicture}
\definecolor{crimson2143940}{RGB}{214,39,40}
\definecolor{darkgray176}{RGB}{176,176,176}
\definecolor{darkorange25512714}{RGB}{255,127,14}
\definecolor{forestgreen4416044}{RGB}{44,160,44}
\definecolor{steelblue31119180}{RGB}{31,119,180}

\begin{axis}[
width=.8\textwidth,
height=.3\textheight,
legend cell align={left},
legend style={
  fill opacity=0,
  draw opacity=1,
  text opacity=1,
  at={(0.03,0.97)},
  anchor=north west,
  draw=none,
  fill=none
},
tick align=outside,
tick pos=left,
x grid style={darkgray176},
xlabel={J},
xmin=-0.25, xmax=49.25,
xtick style={color=black},
y grid style={darkgray176},
ylabel={Relative Frobneius error},
ymin=0.111957284935144, ymax=0.660497049618365,
ytick style={color=black}
]
\addplot [semithick, steelblue31119180]
table {%
2 0.145387378443483
3 0.185809181322044
4 0.20524512299744
5 0.221450744310641
6 0.212007923153292
7 0.209644289961853
8 0.219336511033843
9 0.201931737085048
10 0.213710645359318
11 0.208696491111133
12 0.2039869136722
13 0.200118026642689
14 0.212381744291631
15 0.219803336602282
16 0.24396661886433
17 0.390392046042944
18 0.486304743463625
19 0.534528969070015
20 0.566721367836344
21 0.593808855913495
22 0.615200922571464
23 0.623206094579313
};
\addlegendentry{I=6}
\addplot [semithick, darkorange25512714]
table {%
2 0.136890910602563
3 0.168428637208387
4 0.204818787525566
5 0.223773151887591
6 0.221308396308609
7 0.20502946166791
8 0.214582962628552
9 0.210126596721
10 0.209662227962315
11 0.213104203323558
12 0.201139228860879
13 0.191077597524385
14 0.191260239214439
15 0.188993942652737
16 0.188025302889059
17 0.195961026552129
18 0.194280513004402
19 0.199671910186405
20 0.203601170001871
21 0.212121032012881
22 0.233752991459496
23 0.385204102400837
24 0.46679101458879
25 0.519643063847087
26 0.556705198340671
27 0.582211682242352
28 0.602731683139316
29 0.619124297942016
30 0.627363469220588
};
\addlegendentry{I=7}
\addplot [semithick, forestgreen4416044]
table {%
2 0.150537487967115
3 0.177255582599077
4 0.191679498354716
5 0.222620565754397
6 0.224809040011859
7 0.227404039712459
8 0.220683939616843
9 0.216163439564011
10 0.211660798096176
11 0.207135961975666
12 0.20378281075693
13 0.197535383204459
14 0.201536396196071
15 0.193170104308813
16 0.193516832742613
17 0.180864210903043
18 0.177688873338696
19 0.178417905273963
20 0.17778112343251
21 0.171363586691642
22 0.178333491261682
23 0.175042047975705
24 0.176389239910187
25 0.185997105376104
26 0.187866270024625
27 0.194557764333003
28 0.199182214152687
29 0.242830102677492
30 0.367559797910577
31 0.459240890796228
32 0.508697259203006
33 0.54234349127882
34 0.569781119184918
35 0.591234704283331
36 0.608079014698631
37 0.623835618660949
38 0.632271017806061
};
\addlegendentry{I=8}
\addplot [semithick, crimson2143940]
table {%
2 0.154098349937965
3 0.178055859780847
4 0.210228748287856
5 0.224734350932733
6 0.227580114895574
7 0.229008894023148
8 0.225482837414038
9 0.225676294892151
10 0.219601255511756
11 0.211801677321005
12 0.215279556440537
13 0.206165319722315
14 0.205451320565404
15 0.2033601563205
16 0.187058892647757
17 0.192119803745888
18 0.189361552035611
19 0.186991455982058
20 0.18243396459582
21 0.173489431884708
22 0.162254488060737
23 0.170304205720824
24 0.171943059244799
25 0.167547124081339
26 0.164718872461354
27 0.163353418456279
28 0.163149213978865
29 0.165374196309631
30 0.162469979720949
31 0.173213414385996
32 0.173650500750439
33 0.168866350775306
34 0.185901612599061
35 0.189286476164313
36 0.197700764784188
37 0.25879390378664
38 0.359677176023156
39 0.442094870160824
40 0.495652828071384
41 0.531101910554152
42 0.559052284548572
43 0.579553118469325
44 0.59632038991438
45 0.60978772324884
46 0.628920075413872
47 0.635563423950946
};
\addlegendentry{I=9}
\addplot [semithick, black, dashed, forget plot]
table {%
16 0.111957284935144
16 0.660497049618365
};
\addplot [semithick, black, dashed, forget plot]
table {%
22 0.111957284935144
22 0.660497049618365
};
\addplot [semithick, black, dashed, forget plot]
table {%
28 0.111957284935144
28 0.660497049618365
};
\addplot [semithick, black, dashed, forget plot]
table {%
36 0.111957284935144
36 0.660497049618365
};
\end{axis}

\end{tikzpicture}%
\vspace*{-10mm}
\caption{
Relative Frobenius error using population cumulants. We fix the fourth cumulant of the non-Gaussian sources to be $6$, the second cumulant to be $1$, and consider a standard Gaussian as the Gaussian source. We run 1000 experiments on each pair $(I,J)$ and plot the mean relative Frobenius error. 
The black dashed lines are the identifiability thresholds 
 from Theorem~\ref{thm: real generic matrix result}:
${I\choose 2}+1$ for $I=6,7$ and ${I\choose 2}$ for $I=8,9$. The errors are low for $J$ below the threshold and increase beyond it. The small increase in error from $J={I\choose 2}$ to ${I\choose 2}+1$ for $I=6,7$ is due to the positive probability of non-identifiability when $J = {I \choose 2}+1$, see Theorem~\ref{thm: real generic matrix result}.
    } 
    \label{fig:accurate tensor 6,7,8,9}
\end{figure}
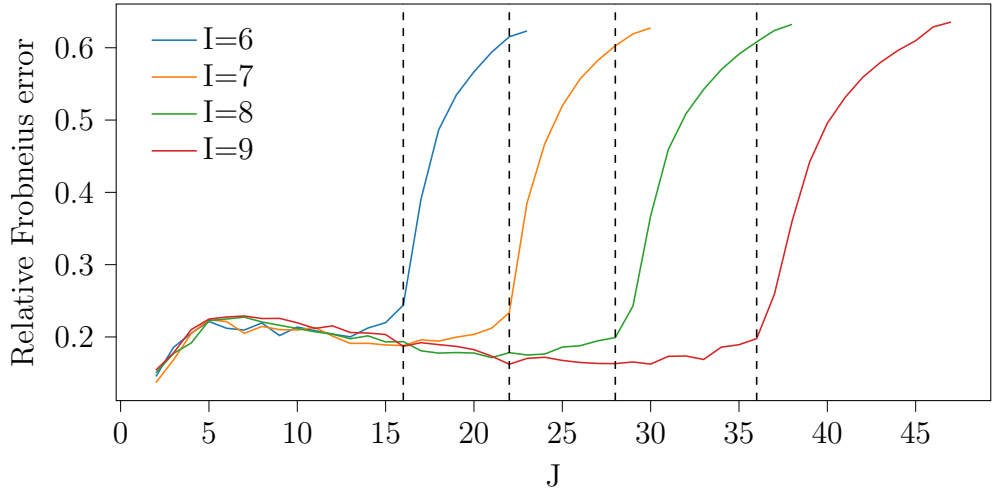

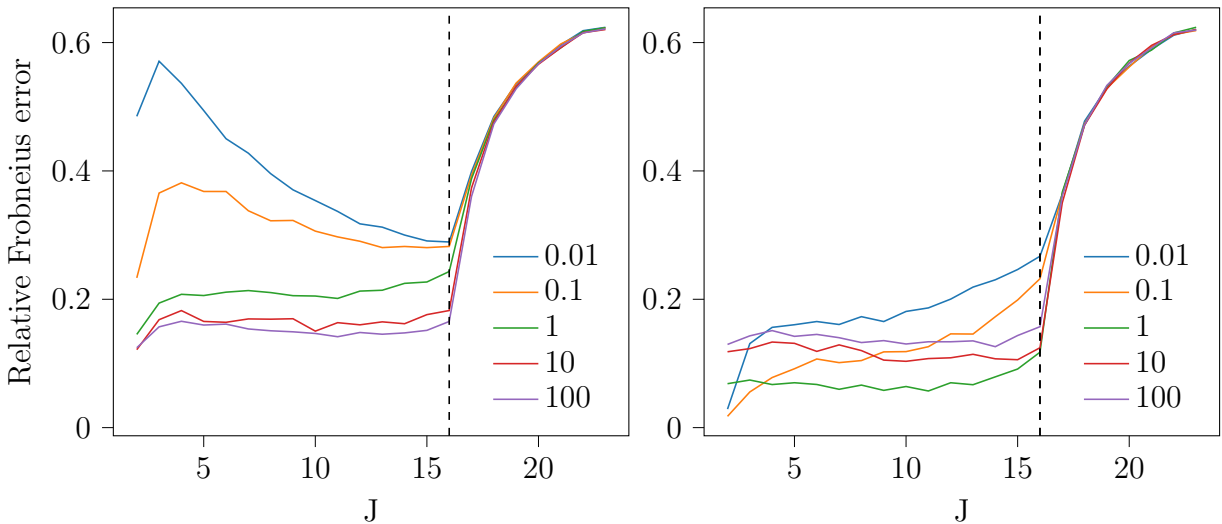
\begin{figure}[htbp]
\centering
\begin{tikzpicture}

\definecolor{crimson2143940}{RGB}{214,39,40}
\definecolor{darkgray176}{RGB}{176,176,176}
\definecolor{darkorange25512714}{RGB}{255,127,14}
\definecolor{forestgreen4416044}{RGB}{44,160,44}
\definecolor{mediumpurple148103189}{RGB}{148,103,189}
\definecolor{steelblue31119180}{RGB}{31,119,180}

\begin{groupplot}[group style={group size=2 by 1}]
\nextgroupplot[
legend cell align={left},
legend style={
  fill opacity=0,
  draw opacity=1,
  text opacity=1,
  at={(0.97,0.03)},
  anchor=south east,
  draw=none,
  fill=none
},
tick align=outside,
tick pos=left,
x grid style={darkgray176},
xlabel={J},
xmin=0.95, xmax=24.05,
xtick style={color=black},
y grid style={darkgray176},
ylabel={Relative Frobneius error},
ymin=-0.0126949994851355, ymax=0.654138500269967,
ytick style={color=black}
]
\addplot [semithick, steelblue31119180]
table {%
2 0.485075146159429
3 0.570919705851576
4 0.536481229695846
5 0.494084701427796
6 0.450217632528733
7 0.427643236616655
8 0.395730706104368
9 0.370540103294744
10 0.353833743407811
11 0.336929195763671
12 0.317548675487933
13 0.312387564953184
14 0.300075627344828
15 0.290919489493005
16 0.289280275704571
17 0.399007107462451
18 0.484731820621217
19 0.535464379005961
20 0.568024354054868
21 0.596393989974503
22 0.618548726555697
23 0.623734119286179
};
\addlegendentry{0.01}
\addplot [semithick, darkorange25512714]
table {%
2 0.233350684128639
3 0.365534016198272
4 0.381462775727403
5 0.367979544561712
6 0.36796343488972
7 0.337937233338154
8 0.322325781312089
9 0.322839019067869
10 0.306194931146727
11 0.297205301824822
12 0.290410454286757
13 0.280531449455808
14 0.282287304351144
15 0.280483021937442
16 0.282317249452874
17 0.393100381202188
18 0.482274363460378
19 0.536916788000484
20 0.569840537504973
21 0.597818907876277
22 0.615125006081595
23 0.623112943088028
};
\addlegendentry{0.1}
\addplot [semithick, forestgreen4416044]
table {%
2 0.14525388873828
3 0.194034043513459
4 0.207736726452618
5 0.205832888765238
6 0.211084155369976
7 0.213566341099864
8 0.210440658455332
9 0.205619790234202
10 0.205054157381134
11 0.201300183061979
12 0.212749120185266
13 0.214078402718499
14 0.224803388303991
15 0.226938591945641
16 0.24325828110089
17 0.387834443143945
18 0.479014928579021
19 0.530265818846191
20 0.568221969402295
21 0.591947349382923
22 0.617460509590892
23 0.623180265367232
};
\addlegendentry{1}
\addplot [semithick, crimson2143940]
table {%
2 0.121482518985189
3 0.168197064741167
4 0.182382113159046
5 0.165464530668103
6 0.164138642118068
7 0.169382804825746
8 0.168976334873508
9 0.169670164034123
10 0.150350005382463
11 0.163556100392898
12 0.160197857223004
13 0.164825475414204
14 0.161912033876978
15 0.176073837569772
16 0.182458132089524
17 0.373596755901929
18 0.476072462938925
19 0.531782808772386
20 0.566202120228035
21 0.59187507104586
22 0.615451638060202
23 0.620311240987113
};
\addlegendentry{10}
\addplot [semithick, mediumpurple148103189]
table {%
2 0.124439041469486
3 0.157042847174504
4 0.165717085047204
5 0.15984944146951
6 0.161285579129698
7 0.153754496298591
8 0.150922460976075
9 0.149298087207749
10 0.146728716769572
11 0.14164679462374
12 0.148240499300604
13 0.145547080199935
14 0.147593672867144
15 0.151569886671915
16 0.165412482430476
17 0.360288280720646
18 0.473073469057897
19 0.527996647955313
20 0.56696840730981
21 0.594301860921081
22 0.61505416056494
23 0.62139779886463
};
\addlegendentry{100}
\addplot [semithick, black, dashed, forget plot]
table {%
16 -0.0126949994851355
16 0.648846699301228
};

\nextgroupplot[
legend cell align={left},
legend style={
  fill opacity=0,
  draw opacity=1,
  text opacity=1,
  at={(0.97,0.03)},
  anchor=south east,
  draw=none,
  fill=none
},
tick align=outside,
tick pos=left,
x grid style={darkgray176},
xlabel={J},
xmin=0.95, xmax=24.05,
xtick style={color=black},
y grid style={darkgray176},
ylabel={},
ymin=-0.0126949994851355, ymax=0.654138500269967,
ytick style={color=black}
]
\addplot [semithick, steelblue31119180]
table {%
2 0.0288336147085925
3 0.130977517556425
4 0.156240513911982
5 0.160350093736156
6 0.165317678018846
7 0.160700506435806
8 0.172904317334738
9 0.165299231125393
10 0.181046066138317
11 0.186554169807606
12 0.200257018854578
13 0.218962777721186
14 0.230367634678683
15 0.246356823389438
16 0.267022918662189
17 0.363587362007187
18 0.47756068858675
19 0.529845001130916
20 0.563328964056284
21 0.58931056286527
22 0.611524846840344
23 0.620460074028873
};
\addlegendentry{0.01}
\addplot [semithick, darkorange25512714]
table {%
2 0.0176156141400964
3 0.0555985808516048
4 0.0780503477492498
5 0.0917479954559651
6 0.106964284323515
7 0.101275651927435
8 0.104500400572899
9 0.118030718257494
10 0.118451509898987
11 0.126093258566204
12 0.146090478537078
13 0.145819604448456
14 0.172837597884953
15 0.198901460864197
16 0.231949841374711
17 0.363955318004253
18 0.473745823614549
19 0.52891658745366
20 0.561994822410905
21 0.593266237656939
22 0.612739207701981
23 0.618722859758112
};
\addlegendentry{0.1}
\addplot [semithick, forestgreen4416044]
table {%
2 0.0684568034195205
3 0.0741731721605066
4 0.0671300537312444
5 0.0699350813722846
6 0.067243415424956
7 0.0597318780117873
8 0.0662855624791515
9 0.058058712138559
10 0.0640601320982008
11 0.0570920770937305
12 0.0699320527198441
13 0.0668693527442032
14 0.0791178945099639
15 0.091316859677871
16 0.11746494711681
17 0.366871038300395
18 0.471996742630558
19 0.531694637026935
20 0.571998964863445
21 0.588302504326643
22 0.614997819348138
23 0.623827886644736
};
\addlegendentry{1}
\addplot [semithick, crimson2143940]
table {%
2 0.118221484640744
3 0.123126893688326
4 0.133349385073247
5 0.131350689775493
6 0.118815036949334
7 0.128929916223627
8 0.120042762617375
9 0.105215341599587
10 0.103339032856433
11 0.107594853982265
12 0.108802967437209
13 0.114254963365369
14 0.107233888123479
15 0.105833470482668
16 0.124458043991682
17 0.350361534141096
18 0.471445920873902
19 0.527806975154539
20 0.568091107957173
21 0.595605364938332
22 0.612375835978144
23 0.620399066242941
};
\addlegendentry{10}
\addplot [semithick, mediumpurple148103189]
table {%
2 0.12964957985106
3 0.1430116783412
4 0.151196256922618
5 0.142232650168885
6 0.14525682415808
7 0.140140449508131
8 0.132569089388158
9 0.135588706094836
10 0.130330013540398
11 0.133791344446971
12 0.133790538877934
13 0.135179557720462
14 0.126140178701993
15 0.143523085390094
16 0.157242154643071
17 0.360965234385827
18 0.472064869605593
19 0.533307363249728
20 0.567468274248321
21 0.591608750605863
22 0.615239770491082
23 0.619417250841175
};
\addlegendentry{100}
\addplot [semithick, black, dashed, forget plot]
table {%
16 -0.0126949994851355
16 0.654138500269967
};
\addplot [semithick, black, dashed, forget plot]
table {%
16 -0.0126949994851355
16 0.654138500269967
};
\end{groupplot}

\end{tikzpicture}
\vspace*{-15mm}
    \caption{Relative Frobenius error for differing Gaussian source variance. We consider variances in the range $\{ 0.01, 0.1, 1, 10, 100 \}$. We fix $I=6$. The black dashed lines are the threshold $J={I\choose 2}+1=16$. For each matrix size and variance, we run the experiment 1000 times and plot the mean. 
    As the variance of the Gaussian source increases, the relative Frobneius error decreases.
In the left figure, we use $1000(I+J)$ iterations in Powell's method. On the right, we increase the number of iterations to $500000$, which makes the algorithm more stable to change of variance.}
    \label{fig:accurate tensor different variances}
\end{figure}

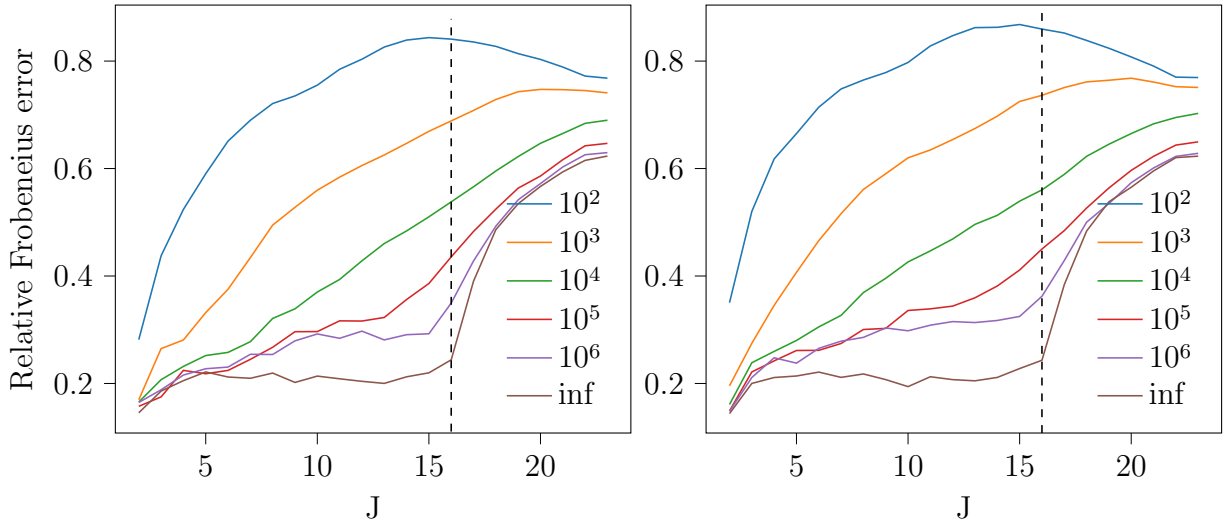
\begin{figure}[htbp]
    \centering
\begin{tikzpicture}

\definecolor{crimson2143940}{RGB}{214,39,40}
\definecolor{darkgray176}{RGB}{176,176,176}
\definecolor{darkorange25512714}{RGB}{255,127,14}
\definecolor{forestgreen4416044}{RGB}{44,160,44}
\definecolor{mediumpurple148103189}{RGB}{148,103,189}
\definecolor{sienna1408675}{RGB}{140,86,75}
\definecolor{steelblue31119180}{RGB}{31,119,180}

\begin{groupplot}[group style={group size=2 by 1}]
\nextgroupplot[
legend cell align={left},
legend style={
  fill opacity=0,
  draw opacity=1,
  text opacity=1,
  at={(0.97,0.03)},
  anchor=south east,
  draw=none,
  fill=none
},
tick align=outside,
tick pos=left,
x grid style={darkgray176},
xlabel={J},
xmin=0.95, xmax=24.05,
xtick style={color=black},
y grid style={darkgray176},
ylabel={Relative Frobeneius error},
ymin=0.107773247047586, ymax=0.904244212483377,
ytick style={color=black}
]
\addplot [semithick, steelblue31119180]
table {%
2 0.281547401182566
3 0.437468341267125
4 0.523949514571982
5 0.590318278378558
6 0.651088692407808
7 0.68984793973642
8 0.721011028136284
9 0.735148737463231
10 0.75518822815885
11 0.784744988504688
12 0.803501254806508
13 0.826120688351022
14 0.83908594784339
15 0.843632053157477
16 0.840899670977242
17 0.835547876970019
18 0.827433496804018
19 0.813868378150561
20 0.803120555975036
21 0.788889727627306
22 0.772116621529183
23 0.768176951988246
};
\addlegendentry{$10^2$}
\addplot [semithick, darkorange25512714]
table {%
2 0.169627149734934
3 0.264799168692668
4 0.280956881760508
5 0.331879890396616
6 0.375363249523798
7 0.434167715516405
8 0.494780567107104
9 0.527833638776256
10 0.559686813347674
11 0.583954113520202
12 0.605564667914989
13 0.625216864546349
14 0.646870918991694
15 0.669298146561541
16 0.688663231015739
17 0.708163392366625
18 0.728527849253589
19 0.742968442693434
20 0.747292446980467
21 0.746799489933978
22 0.745099627432696
23 0.740832702107131
};
\addlegendentry{$10^3$}
\addplot [semithick, forestgreen4416044]
table {%
2 0.165440139896478
3 0.207009554479928
4 0.231740177470491
5 0.251949760599113
6 0.257823236395572
7 0.277738818896748
8 0.321019376630982
9 0.338908567018739
10 0.370188170007442
11 0.393569621171288
12 0.428270288310797
13 0.46041146728276
14 0.483908189769646
15 0.510124581752028
16 0.537979844587704
17 0.566055084452909
18 0.59544392510405
19 0.622247759154525
20 0.64698155036168
21 0.66526769627989
22 0.684108370187937
23 0.689951648635497
};
\addlegendentry{$10^4$}
\addplot [semithick, crimson2143940]
table {%
2 0.157475496757222
3 0.174997691866253
4 0.224355311935554
5 0.21810915408106
6 0.224297791066187
7 0.244879219520515
8 0.267378785484706
9 0.296339925610774
10 0.296556504686381
11 0.316547287335781
12 0.316163149473091
13 0.322829780374575
14 0.356216542055977
15 0.386057704063067
16 0.435574172816398
17 0.482892259894373
18 0.524429886423367
19 0.563550086868051
20 0.586066325819098
21 0.616480541625539
22 0.642340003809462
23 0.646856803429395
};
\addlegendentry{$10^5$}
\addplot [semithick, mediumpurple148103189]
table {%
2 0.16466434165407
3 0.188211134954589
4 0.21568655081137
5 0.227400158527214
6 0.230506041890233
7 0.254270936790947
8 0.25389065245664
9 0.279325446789089
10 0.29213048294884
11 0.284066191803526
12 0.297375295122918
13 0.28105117581634
14 0.290775330398967
15 0.292481309073876
16 0.349596684747674
17 0.428060396937173
18 0.49280885126648
19 0.541825025989931
20 0.572022016667034
21 0.602949455190661
22 0.625688875543057
23 0.629648265962641
};
\addlegendentry{$10^6$}
\addplot [semithick, sienna1408675]
table {%
2 0.145387378443483
3 0.185809181322044
4 0.20524512299744
5 0.221450744310641
6 0.212007923153292
7 0.209644289961853
8 0.219336511033843
9 0.201931737085048
10 0.213710645359318
11 0.208696491111133
12 0.2039869136722
13 0.200118026642689
14 0.212381744291631
15 0.219803336602282
16 0.24396661886433
17 0.390392046042944
18 0.486304743463625
19 0.534528969070015
20 0.566721367836344
21 0.593808855913495
22 0.615200922571464
23 0.623206094579313
};
\addlegendentry{inf}
\addplot [semithick, black, dashed, forget plot]
table {%
16 0.110475144707783
16 0.878544286893177
};

\nextgroupplot[
legend cell align={left},
legend style={
  fill opacity=0,
  draw opacity=1,
  text opacity=1,
  at={(0.97,0.03)},
  anchor=south east,
  draw=none,
  fill=none
},
tick align=outside,
tick pos=left,
x grid style={darkgray176},
xlabel={J},
xmin=0.95, xmax=24.05,
xtick style={color=black},
y grid style={darkgray176},
ymin=0.107773247047586, ymax=0.904244212483377,
ytick style={color=black}
]
\addplot [semithick, steelblue31119180]
table {%
2 0.350384669496495
3 0.520191317699643
4 0.617890866188145
5 0.665074239884582
6 0.714183087863518
7 0.748141846931324
8 0.764734042331409
9 0.778630625494191
10 0.797450533773379
11 0.828190340308499
12 0.847250959516266
13 0.862269848514302
14 0.862846818061725
15 0.86804098678175
16 0.859573359104069
17 0.852226559417823
18 0.838449439127905
19 0.823871279413514
20 0.80770836027117
21 0.790720287645084
22 0.77011350024683
23 0.769357151768388
};
\addlegendentry{$10^2$}
\addplot [semithick, darkorange25512714]
table {%
2 0.195676842537254
3 0.274727633249658
4 0.345418912950182
5 0.406815046147054
6 0.465815592053093
7 0.51602956823059
8 0.561150475247919
9 0.590303410404526
10 0.619924220843823
11 0.634596050297171
12 0.653907081292977
13 0.674547373172215
14 0.697497098961796
15 0.724818970760738
16 0.736254886679988
17 0.75075158841807
18 0.761278413936608
19 0.764306213413659
20 0.768136650290713
21 0.760995138376522
22 0.752330061486968
23 0.750937696254655
};
\addlegendentry{$10^3$}
\addplot [semithick, forestgreen4416044]
table {%
2 0.160826918863482
3 0.23859080158249
4 0.259677857381942
5 0.280031121035092
6 0.305230555991331
7 0.326740698840926
8 0.369217781120563
9 0.395708304805236
10 0.426217226175902
11 0.446795693379249
12 0.468856327701749
13 0.496104564384698
14 0.512943125637539
15 0.539279936583048
16 0.560242799170367
17 0.58913102281714
18 0.622717844214972
19 0.645253499273481
20 0.665117568904168
21 0.683176108540841
22 0.694989431310803
23 0.702642640452127
};
\addlegendentry{$10^4$}
\addplot [semithick, crimson2143940]
table {%
2 0.148940793601683
3 0.221425494322871
4 0.242041231785951
5 0.261355872999342
6 0.261702555245031
7 0.274397161579997
8 0.300472737280771
9 0.302482764029205
10 0.335849812194123
11 0.338855840816844
12 0.343964966819182
13 0.359483395300441
14 0.381401080257396
15 0.411157451344757
16 0.450218244950585
17 0.484426731221003
18 0.526425655524716
19 0.563802541123194
20 0.596584837955792
21 0.622720308924102
22 0.64364694415076
23 0.6496495908488
};
\addlegendentry{$10^5$}
\addplot [semithick, mediumpurple148103189]
table {%
2 0.14830661207429
3 0.211145654821293
4 0.24782199371776
5 0.237833585937628
6 0.265348594359236
7 0.278709111988457
8 0.285778825594445
9 0.30328325150041
10 0.298060286191346
11 0.308472749009315
12 0.315159274838457
13 0.313376706596576
14 0.317228930412068
15 0.324751961080223
16 0.362246778637271
17 0.428736338042853
18 0.500180091939682
19 0.534788745400149
20 0.573637931044198
21 0.601023744847728
22 0.622612905136902
23 0.628339659814555
};
\addlegendentry{$10^6$}
\addplot [semithick, sienna1408675]
table {%
2 0.143976472749213
3 0.200078865804726
4 0.211063101247569
5 0.213885993626632
6 0.221155193596539
7 0.211351687936054
8 0.217685041955173
9 0.207172404480418
10 0.194144474289826
11 0.212510042986819
12 0.207133999648566
13 0.204917978696713
14 0.211263759214826
15 0.227758837645344
16 0.243184811628038
17 0.383843709286763
18 0.483652134442813
19 0.53811011401435
20 0.565183193039986
21 0.595525492698038
22 0.620345137849353
23 0.623040037954855
};
\addlegendentry{inf}
\addplot [semithick, black, dashed, forget plot]
table {%
16 0.107773247047586
16 0.904244212483377
};
\end{groupplot}

\end{tikzpicture}
\vspace*{-15mm}
    \caption{Relative Frobenius error with sample cumulant tensors. We take our non-Gaussian sources to be exponential sources with parameter 1 (left) and Student $t$-distributed sources with five degrees of freedom (right). We set the Gaussian source to be a standard Gaussian. We fix $I = 6$. For each pair $(I,J)$, we run 1000 experiments and plot the mean Frobneius error. In both plots, the error decreases as the sample size increases. We plot the population cumulant method (labelled as `inf') for comparison. }
    \label{fig:exponential source different samplesizes}
\end{figure}

\subsection{Image data}

We test our algorithm on the CIFAR-10 dataset \cite{cifar10}, following~\cite{podosinnikova2019overcomplete}. We define a training set of 50000 color images, each of size $32\times 32$, in one of 10 classes. We convert each image to grayscale and divided the central $28\times 28$ image into $16$ images of size $7\times 7$. Our assumption is that there is a collection of $7 \times 7$ images, from which the others are expressible as a linear combination.
We use Algorithm~\ref{alg:recover A} to plot the columns of the $49\times J$ mixing matrix, see~\autoref{fig:image data}.

\begin{figure}[htbp]
    \centering
\includegraphics[scale=0.23]{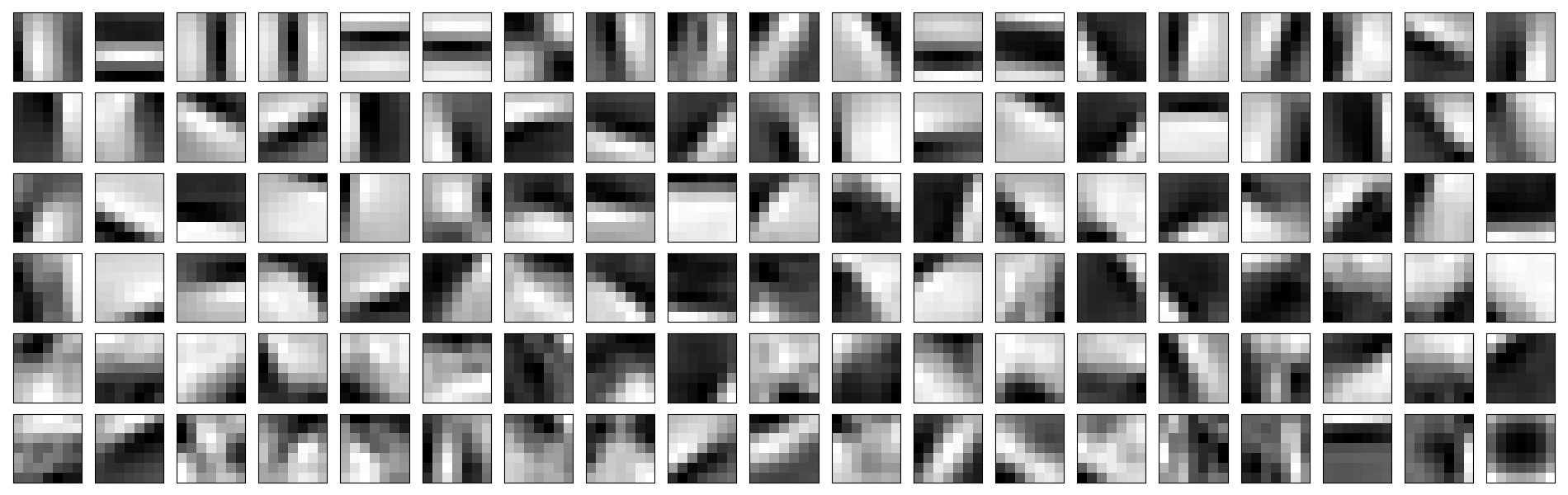}

\bigskip

\includegraphics[scale=0.23]{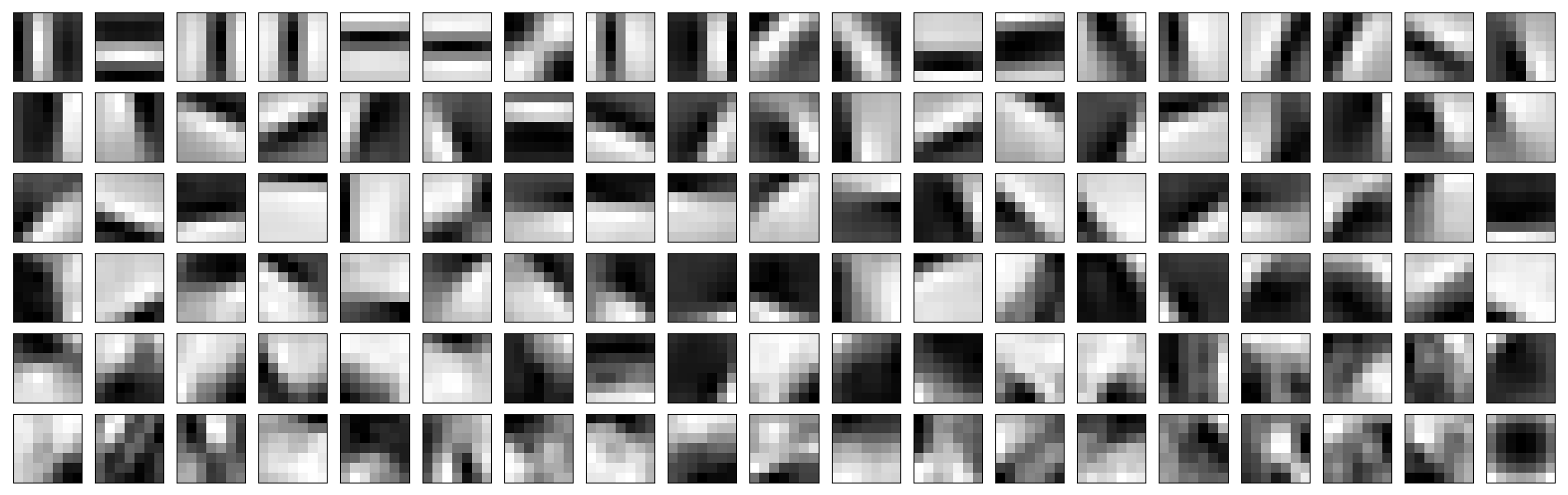} 
    \caption{
    We divide the $7 \times 7$ images in half to give two datasets, each with $400000 = \frac12 (16 \times 50000)$ datapoints of dimension $49$, keeping the number of images from each class roughly the same between the two halves. 
    We apply Algorithm~\ref{alg:recover A} to the two datasets and assess the similarity of the output. We obtain two matrices $A \in \RR^{49\times J}$, where $J$ is the number of sources. We illustrate the results for $J=114$. 
    The columns of the two $49\times 114$ matrices are plotted as grayscale $7 \times 7$ images. We observe the visual agreement of the $114$ images, reflecting the identifiability. 
    The last image is the Gaussian source, the Gaussian noise in the images. The two Gaussian sources have cosine similarity $0.99$. Their grayscale plots show that pixels patterns have more Gaussian noise at the center than the edges.
    }
    \label{fig:image data}
\end{figure}

\subsection{Protein data}

We fit an adapted LiNGAM model~\cite{shimizu2006linear} to a single-cell flow cytometry dataset~\cite{sachs2005causal}. Each datapoint measures 11 proteins in a cell. Suppose that the 11 proteins are $X_1,\ldots,X_{11}$ and that $G$ is a directed acyclic graph with nodes $X_1,\ldots,X_{11}$ whose edges $E(G)$ indicates causal relationships with weights $\lambda_{ij}$ on the edge $j \to i$. 
A linear structural equation model writes
$$X_i=\sum_{(j\to i) \in E(G)} \lambda_{ij}X_j+e_i.$$
The LiNGAM algorithm learns the graph $G$ from the higher-order cumulants of $X$, assuming the noise terms $e_i$ are non-Gaussian, using ICA. 
We use our algorithm for ICA with a Gaussian source to adapt the LiNGAM to allow a latent source of Gaussian noise:
$$X_i=\sum_{(j\to i) \in E(G)} \lambda_{ij}X_j+e_i+t_i y,$$
where $y$ is a Gaussian variable and $t_i$ is its effect on variable $X_i$.
Let $\Lambda \in \RR^{11 \times 11}$ be the matrix of weights, with $(i,j)$ entry $\lambda_{ij}$. 
Then 
$$
\mathbf{X}  = \Lambda \mathbf{X} + \mathbf{e} + \mathbf{t} y \quad 
\implies \quad \mathbf{X} =(I-\Lambda)^{-1}\mathbf{e}+(I-\Lambda)^{-1}\mathbf{t}y=
    ((I-\Lambda)^{-1} |(I-\Lambda)^{-1}\mathbf{t})\begin{pmatrix}
        \mathbf{e}\\
        y
    \end{pmatrix}.
    $$
    Algorithm \ref{alg:recover A}
recovers $A \in \RR^{11 \times 12}$ with first $11$ columns $(I-\Lambda)^{-1}$ and last column $(I-\Lambda)^{-1}\mathbf{t}$.
As in the LiNGAM algorithm, this enables us to recover the directed acyclic graph $G$. But we also recover the vector $\mathbf{t}$, which measures the Gaussian noise effect for each $X_i$. 
 
There are samples collected under 13 different perturbations in~\cite{sachs2005causal}. To test our adapted LiNGAM model via Algorithm~\ref{alg:recover A}, we divide the data from each perturbation into half to form two datasets. We log transform the data. 
We run our algorithm on each half of the data and compute the similarity of the Gaussian source effects, using cosine similarity, see Figure~\ref{fig:protein data}.
We plot the Gaussian source effects in Figure~\ref{fig:bar}. 

\bigskip 

\begin{figure}[htbp]
    \centering
\begin{tikzpicture}

\definecolor{darkgray176}{RGB}{176,176,176}
\definecolor{steelblue31119180}{RGB}{31,119,180}

\begin{groupplot}[group style={group size=2 by 1}]
\nextgroupplot[
xlabel={Cosine similarity},
tick align=outside,
tick pos=left,
x grid style={darkgray176},
xmin=0.9, 
xmax=1.005,
xtick style={color=black},
y grid style={darkgray176},
ylabel={Count},
ymin=0, ymax=217.35,
ytick style={color=black}
]
\draw[draw=none,fill=steelblue31119180] (axis cs:0.915388084569843,0) rectangle (axis cs:0.919551880014173,8);
\draw[draw=none,fill=steelblue31119180] (axis cs:0.919551880014173,0) rectangle (axis cs:0.923715675458504,37);
\draw[draw=none,fill=steelblue31119180] (axis cs:0.923715675458504,0) rectangle (axis cs:0.927879470902835,42);
\draw[draw=none,fill=steelblue31119180] (axis cs:0.927879470902835,0) rectangle (axis cs:0.932043266347165,22);
\draw[draw=none,fill=steelblue31119180] (axis cs:0.932043266347165,0) rectangle (axis cs:0.936207061791496,25);
\draw[draw=none,fill=steelblue31119180] (axis cs:0.936207061791496,0) rectangle (axis cs:0.940370857235827,19);
\draw[draw=none,fill=steelblue31119180] (axis cs:0.940370857235827,0) rectangle (axis cs:0.944534652680157,18);
\draw[draw=none,fill=steelblue31119180] (axis cs:0.944534652680157,0) rectangle (axis cs:0.948698448124488,30);
\draw[draw=none,fill=steelblue31119180] (axis cs:0.948698448124488,0) rectangle (axis cs:0.952862243568819,66);
\draw[draw=none,fill=steelblue31119180] (axis cs:0.952862243568819,0) rectangle (axis cs:0.957026039013149,102);
\draw[draw=none,fill=steelblue31119180] (axis cs:0.957026039013149,0) rectangle (axis cs:0.96118983445748,31);
\draw[draw=none,fill=steelblue31119180] (axis cs:0.96118983445748,0) rectangle (axis cs:0.965353629901811,1);
\draw[draw=none,fill=steelblue31119180] (axis cs:0.965353629901811,0) rectangle (axis cs:0.969517425346141,2);
\draw[draw=none,fill=steelblue31119180] (axis cs:0.969517425346141,0) rectangle (axis cs:0.973681220790472,9);
\draw[draw=none,fill=steelblue31119180] (axis cs:0.973681220790472,0) rectangle (axis cs:0.977845016234803,39);
\draw[draw=none,fill=steelblue31119180] (axis cs:0.977845016234803,0) rectangle (axis cs:0.982008811679134,49);
\draw[draw=none,fill=steelblue31119180] (axis cs:0.982008811679133,0) rectangle (axis cs:0.986172607123464,83);
\draw[draw=none,fill=steelblue31119180] (axis cs:0.986172607123464,0) rectangle (axis cs:0.990336402567795,167);
\draw[draw=none,fill=steelblue31119180] (axis cs:0.990336402567795,0) rectangle (axis cs:0.994500198012126,207);
\draw[draw=none,fill=steelblue31119180] (axis cs:0.994500198012126,0) rectangle (axis cs:0.998663993456456,43);

\nextgroupplot[
xlabel={Cosine similarity},
tick align=outside,
tick pos=left,
x grid style={darkgray176},
xmin=-0.038908745363811, 
xmax=1.05,
xtick style={color=black},
y grid style={darkgray176},
ymin=0, ymax=105,
ytick style={color=black}
]
\draw[draw=none,fill=steelblue31119180] (axis cs:8.96147985839819e-05,0) rectangle (axis cs:0.039087974960979,96);
\draw[draw=none,fill=steelblue31119180] (axis cs:0.039087974960979,0) rectangle (axis cs:0.078086335123374,100);
\draw[draw=none,fill=steelblue31119180] (axis cs:0.078086335123374,0) rectangle (axis cs:0.117084695285769,80);
\draw[draw=none,fill=steelblue31119180] (axis cs:0.117084695285769,0) rectangle (axis cs:0.156083055448164,99);
\draw[draw=none,fill=steelblue31119180] (axis cs:0.156083055448164,0) rectangle (axis cs:0.195081415610559,81);
\draw[draw=none,fill=steelblue31119180] (axis cs:0.195081415610559,0) rectangle (axis cs:0.234079775772954,69);
\draw[draw=none,fill=steelblue31119180] (axis cs:0.234079775772954,0) rectangle (axis cs:0.273078135935349,63);
\draw[draw=none,fill=steelblue31119180] (axis cs:0.273078135935349,0) rectangle (axis cs:0.312076496097744,69);
\draw[draw=none,fill=steelblue31119180] (axis cs:0.312076496097744,0) rectangle (axis cs:0.351074856260139,62);
\draw[draw=none,fill=steelblue31119180] (axis cs:0.351074856260139,0) rectangle (axis cs:0.390073216422534,51);
\draw[draw=none,fill=steelblue31119180] (axis cs:0.390073216422534,0) rectangle (axis cs:0.429071576584929,57);
\draw[draw=none,fill=steelblue31119180] (axis cs:0.429071576584929,0) rectangle (axis cs:0.468069936747324,42);
\draw[draw=none,fill=steelblue31119180] (axis cs:0.468069936747324,0) rectangle (axis cs:0.507068296909719,30);
\draw[draw=none,fill=steelblue31119180] (axis cs:0.507068296909719,0) rectangle (axis cs:0.546066657072114,28);
\draw[draw=none,fill=steelblue31119180] (axis cs:0.546066657072114,0) rectangle (axis cs:0.585065017234509,18);
\draw[draw=none,fill=steelblue31119180] (axis cs:0.585065017234509,0) rectangle (axis cs:0.624063377396904,15);
\draw[draw=none,fill=steelblue31119180] (axis cs:0.624063377396904,0) rectangle (axis cs:0.663061737559299,17);
\draw[draw=none,fill=steelblue31119180] (axis cs:0.663061737559299,0) rectangle (axis cs:0.702060097721694,12);
\draw[draw=none,fill=steelblue31119180] (axis cs:0.702060097721694,0) rectangle (axis cs:0.741058457884089,7);
\draw[draw=none,fill=steelblue31119180] (axis cs:0.741058457884089,0) rectangle (axis cs:0.780056818046484,4);
\end{groupplot}

\end{tikzpicture}
\vspace{-15mm}
    \caption{For 1000 experiments, we plot a histogram of the cosine similarity of the Gaussian column of the mixing matrices (left). For comparison, we show the plot for two random matrices (right). Randomness comes from both tensor decomposition and the optimization step. The plot validates the choice of our adapted LiNGAM model, since the Gaussian source effects are consistent between experiments.} 
    \label{fig:protein data}
\end{figure}
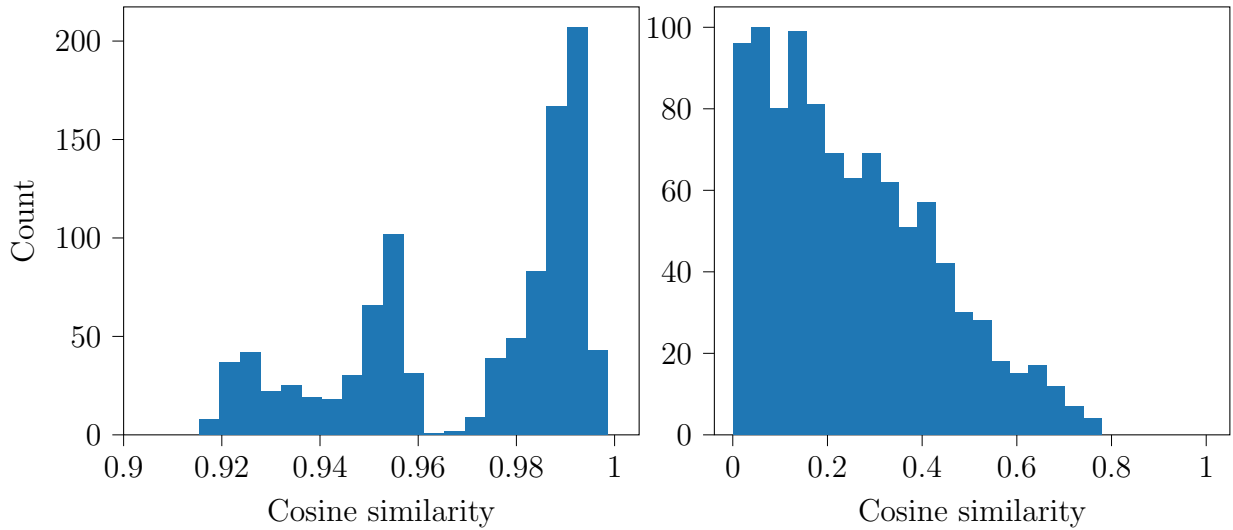

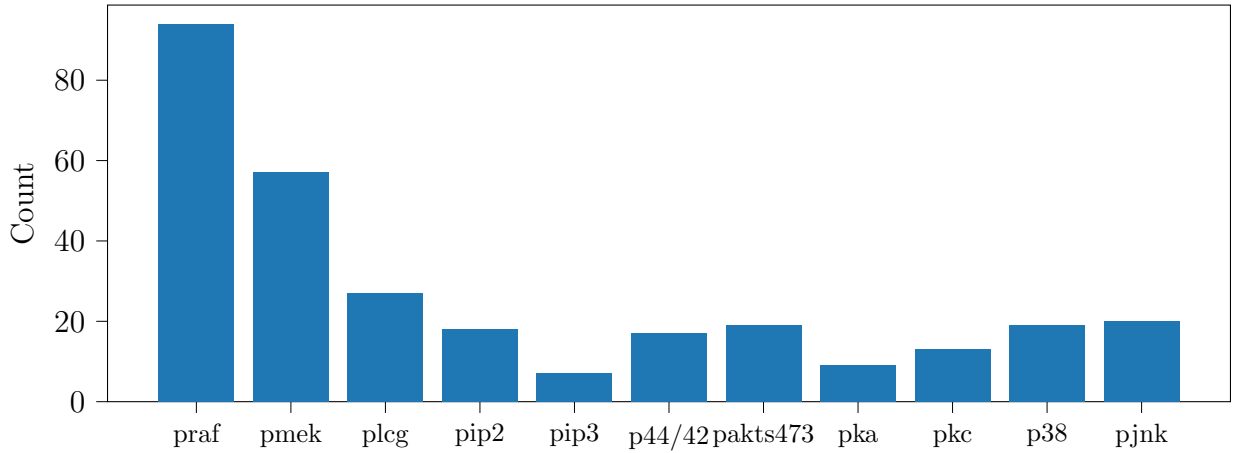
\begin{figure}[!htb]
    \centering
\begin{tikzpicture}

\definecolor{darkgray176}{RGB}{176,176,176}
\definecolor{steelblue31119180}{RGB}{31,119,180}

\begin{axis}[
width=1\textwidth,
height=.3\textheight,
ylabel={Count},
tick align=outside,
tick pos=left,
x grid style={darkgray176},
xmin=-0.94, xmax=10.94,
xtick style={color=black},
xtick={0,1,2,3,4,5,6,7,8,9,10},
xticklabels={\footnotesize praf,\footnotesize pmek,\footnotesize plcg,\footnotesize pip2,\footnotesize pip3,{\footnotesize p44/42}, {\footnotesize pakts473},\footnotesize pka,\footnotesize pkc,\footnotesize p38,\footnotesize pjnk},
y grid style={darkgray176},
ymin=0, ymax=98.7,
ytick style={color=black}
]
\draw[draw=none,fill=steelblue31119180] (axis cs:-0.4,0) rectangle (axis cs:0.4,94);
\draw[draw=none,fill=steelblue31119180] (axis cs:0.6,0) rectangle (axis cs:1.4,57);
\draw[draw=none,fill=steelblue31119180] (axis cs:1.6,0) rectangle (axis cs:2.4,27);
\draw[draw=none,fill=steelblue31119180] (axis cs:2.6,0) rectangle (axis cs:3.4,18);
\draw[draw=none,fill=steelblue31119180] (axis cs:3.6,0) rectangle (axis cs:4.4,7);
\draw[draw=none,fill=steelblue31119180] (axis cs:4.6,0) rectangle (axis cs:5.4,17);
\draw[draw=none,fill=steelblue31119180] (axis cs:5.6,0) rectangle (axis cs:6.4,19);
\draw[draw=none,fill=steelblue31119180] (axis cs:6.6,0) rectangle (axis cs:7.4,9);
\draw[draw=none,fill=steelblue31119180] (axis cs:7.6,0) rectangle (axis cs:8.4,13);
\draw[draw=none,fill=steelblue31119180] (axis cs:8.6,0) rectangle (axis cs:9.4,19);
\draw[draw=none,fill=steelblue31119180] (axis cs:9.6,0) rectangle (axis cs:10.4,20);
\end{axis}

\end{tikzpicture}
\vspace*{-15mm}
    \caption{
        We fit our adapted LiNGAM model and focus on how the Gaussian source effects each protein. We plot
    the number of times each protein appears among the three entries of $\mathbf{t}$ with largest absolute value, across 100 repeats.
    It indicates that the (log-transformed) measurements of praf and pmek contain more Gaussian noise than the other proteins.}
    \label{fig:bar}
\end{figure} 

\section{Conclusion}

In this paper, we characterized the identifiability of overcomplete ICA.
For generic mixing, we saw how identifiability is determined by the number of sources and the number of observations.
We gave an algorithm for recovering the mixing matrix from the second and fourth cumulants and tested it on real and simulated data. Our algorithm allows for a Gaussian source, which is not true of other algorithms for ICA or overcomplete ICA. 

We conclude by mentioning directions for future study.
\begin{itemize}
    \item Theorem~\ref{thm: real generic matrix result} gives three possibilities for generic identifiability: the middle case has a positive probability of identifiability and of non-identifiability. Compute these probabilities for a suitable distribution of mixing matrices. 
    \item Adapt Theorem~\ref{thm: real generic matrix result} and Algorithm~\ref{alg:recover A} to incorporate structure on $A$, such as sparsity. 
    \item Extend the complex identifiability results such as Theorem \ref{thm: when compelx identifiable exists for J>Ichoose 2}to the real setting.
    \item Study the special loci of identifiable and non-identifiable matrices geometrically, e.g. compute the dimension and degree of their Zariski closures.
\end{itemize}

\medskip

\noindent {\bf Acknowledgements.} We thank 
Chiara Meroni, Kristian Ranestad
 and Piotr Zwiernik for helpful discussions. AS was supported by the NSF (DMR-2011754).
 AS would like to thank the Isaac Newton Institute for Mathematical Sciences, Cambridge, for support and hospitality during the programme `New equivariant methods in algebraic and differential geometry' where work on this paper was undertaken. This work was supported by EPSRC grant no EP/R014604/1.
 
\bibliographystyle{alpha}
\bibliography{sample}

\end{document}